\documentclass[12pt]{amsart}
\usepackage{calrsfs}
\usepackage{bookman}
\usepackage[backref]{hyperref}
\usepackage[alphabetic,backrefs,lite,nobysame]{amsrefs} 
\usepackage{graphicx}
\usepackage[OT2,T1]{fontenc}
\usepackage{comment}
\usepackage{fullpage}
\usepackage{fancyhdr}
\usepackage{microtype}

\newcommand{\aaa}{{\mathfrak a}}
\newcommand{\Qq}{{\mathfrak Q}}

\newcommand{\Tt}{{\mathfrak T}}

\newcommand{\Nn}{{\mathfrak N}}
\newcommand{\Mm}{{\mathfrak M}}

\newcommand{\calD}{{\mathcal D}}

\newcommand{\calP}{{\mathcal P}}
\newcommand{\calQ}{{\mathcal Q}}

\newcommand{\calT}{{\mathcal T}}

\newcommand{\ord}{\mbox{ord}}
\newcommand{\pp}{{\mathfrak p}}

\newcommand{\calC}{{\mathcal C}}
\newcommand{\calS}{{\mathcal S}}
\newcommand{\calE}{{\mathcal E}}

\newcommand{\qq}{{\mathfrak q}}
\newcommand{\uu}{{\mathfrak u}}
\newcommand{\ttt}{{\mathfrak t}}

\theoremstyle{plain}

\newtheorem{theorem}{Theorem}[section]
\newtheorem{lemma}[theorem]{\bf Lemma}
\newtheorem{corollary}[theorem]{\bf Corollary}
\newtheorem{proposition}[theorem]{\bf Proposition}

\theoremstyle{definition}
\newtheorem{definition}[theorem]{\bf Definition}
\newtheorem{question}[theorem]{\bf Question}

\theoremstyle{remark}
\newtheorem{remark}[theorem]{\bf Remark}

\newtheorem{notation}[theorem]{\bf Notation}
\newtheorem{notationassumptions}[theorem]{\bf Notation and Assumptions}

\renewcommand{\deg}[1]{\ensuremath{\text{deg}(#1)}}

\newcommand{\N}{\mathbb{N}}
\newcommand{\Z}{\mathbb{Z}}

\newcommand{\Q}{\mathbb{Q}}
\newcommand{\R}{\mathbb{R}}

\newcommand{\QQ}{\mathfrak{Q}}

\newcommand{\Gal}[2]{\text{Gal}(#1/#2)}

\newcommand{\at}{\char'100}
\def\Nn{{\mathfrak N}}

\newcommand{\dom}[1]{\text{dom}(#1)}

\def\s01{\ensuremath{\Sigma^0_1}}
\def\d02{\ensuremath{\Delta^0_2}}
\def\phi{\varphi}

\usepackage{microtype}

\title{On existential definitions of c.e.  subsets of rings of functions of characteristic 0}
\author{Russell Miller
\& Alexandra Shlapentokh}\thanks{The first author was partially supported by Grant \# DMS -- 1362206 from
the National Science Foundation, by Grant \# 581896 from the Simons Foundation,
by the Queens College Research Enhancement Program, and by several grants from
The City University of New York PSC-CUNY Research Award Program.
The second author was partially supported by Grant \# DMS -- 1161456 from
the National Science Foundation.}

\begin{document}

\begin{abstract}
We extend results of Denef, Zahidi, Demeyer and the second author to show the following.
\begin{enumerate}
\item Every c.e.\ set of integers has a single-fold Diophantine definition over the ring of integral functions of any function field of characteristic 0.
\item Every c.e.\ set of integers has a single-fold Diophantine definition over a polynomial ring over integral domain $R$ characteristic 0.
\item All c.e.\ subsets of polynomial rings over rings of totally real integers have finite-fold Diophantine definitions.  (These are the first examples of infinite rings with this property.)
\item Let $K$ be a one-variable function field over a  field of constants $k$, and let $\pp$ be any prime of $K$.  If $k$ is algebraic over $\Q$ and for some odd prime $p$  embeddable into a finite extension of $\Q_{\pp}$,  then the valuation ring of $\pp$ has a Diophantine definition over $K$. If $k$ is embeddable into a real field, then valuation rings are existentially definable for ``almost all'' primes.
\item Let $K$ be a one-variable function field over a number field and let $\calS$ be a finite non-empty set of its primes.  Then all c.e.\ subsets of $O_{K,\calS}$ are Diophantine over $O_{K,\calS}$.  (Here $O_{K,\calS}$ is the ring of $\calS$-integers or a ring of integral functions.)
\end{enumerate}
\end{abstract}

\maketitle

\thispagestyle{empty}
\section{Introduction}
In 1969, building on earlier work by Martin Davis, Hilary Putnam and Julia Robinson, Yuri Matiyasevich demonstrated the impossibility of solving Hilbert's Tenth Problem.    In doing so, he also completed a proof of the theorem asserting that Diophantine (or existentially definable in the language of rings) sets  and computably enumerable sets of integers were the same.  In other words, it was proved that for every positive integer $n$, every computably enumerable subset of $\Z^n$ had a Diophantine definition over $\Z$.
We describe the notions of a Diophantine definition and a Diophantine set in a more general setting.
\begin{definition}
Let $R$ be a commutative ring and let $n$ be a positive integer.  In this case a set $A \subset R^n$ is called Diophantine over $R$ if for some $m>0$ and some polynomial 
\[
f(T_1,\ldots, T_n,X_1,\ldots,X_m) \in R[\bar T,\bar X]
\]
 we have that 
for all $(t_1,\ldots,t_n) \in R^n$ it is the case that
\[
(t_1,\ldots,t_n) \in A \mbox{ if and only if } \exists x_1, \ldots,x_m \in R \mbox{ such that } f(t_1,\ldots,t_n,x_1,\ldots,x_m)=0.
\]
The polynomial $f(\bar T,\bar X)$ is called a Diophantine definition of $A$ over $R$.
\end{definition}
If a set $A$ is Diophantine over $R$ and for every $\bar t \in A$ we have that $\bar x$ as above is unique, we say that $f(\bar T, \bar X)$ is a {\it single-fold} definition of $A$.  If for every $\bar t \in A$ we have that there are only finitely many $\bar x$ as above, we say that $f(\bar T, \bar X)$ is a {\it finite-fold} definition of $A$. 
\begin{question}
Does every c.e.\ set of integers have a finite-fold Diophantine definition over $\Z$?
\end{question}
The answer to this question, raised by Yuri Matiyasevich almost immediately after his solution to Hilbert's Tenth Problem, is unknown to this day.  The issue of finite-fold representation is of more than just esoteric interest because of its connection to many other questions.  For an extensive survey of these connections we refer the reader to a paper of Matiyasevich (\cite{Mat10}).  Here we would like to give just one example that can be considered a generalization of Hilbert's Tenth Problem.  

Let ${\Nn}=\{0,1, \ldots, \aleph_0\}$ and let $\Mm$ be any nonempty proper subset of $\Nn$.  Let $\calP(\Mm)$ be the set of polynomials $P$ with integer coefficients such that the number of solutions to the equation $P=0$ is in $\Mm$. Martin Davis showed in \cite{Da3} that $\calP(\Mm)$ is undecidable.  If we ask whether $\calP(\Mm)$ is c.e, then the answer is currently unknown.  At the same time, if we replace polynomials by exponential Diophantine equations, then we can  answer the question.  Craig Smory\'{n}sky in \cite{Smory} proved that  $\calE(\Mm)$ is c.e.\ if and only if $\Mm=\{\alpha|\alpha \geq \beta\}$ for some finite $\beta$.  (Here $\calE(\Mm)$ is a collection of exponential Diophantine polynomials with positive integer coefficients such that if an exponential Diophantine polynomial $E \in \calE(\Mm)$, then the number of solutions to the equation $E=0$ is in $\Mm$.)  Smory\'{n}sky's proof relied on a result obtained by Matiyasevich  in \cite{Mat77} that every computably enumerable set has a single-fold {\it exponential} Diophantine definition.  One would expect a similar result for (non-exponential) Diophantine equations if the finite-fold question is answered affirmatively.  

Matiyasevich also proved that
to show that all c.e.\ sets of integers have single-fold (or finite-fold) Diophantine definitions it is enough to show that the set of pairs  $\{(a, b) \in \Z_{>0}^2|b=2^a\}$ has a single-fold (finite-fold) Diophantine definition.  (This will not be surprising to readers familiar with the history of Hilbert's Tenth Problem.)

Unfortunately, the finite-fold question over $\Z$ remains out of reach at the moment, as with many other Diophantine questions.  In Section \ref{intsingle} of this  paper, we take some first timid steps in the investigation of this issue by considering it in a more hospitable environment over function fields of characteristic 0, as described  in Section \ref{sec:function}.  We extend the results of the second author from \cite{Sh34} to show that over any ring of integral functions (otherwise known as a ring of $\calS$-integers)  any c.e. set of rational integers has a single-fold Diophantine definition (see Theorem \ref{thm:singlefoldZ}).    We also show that over any polynomial ring over an integral domain $Z$ of characteristic 0 it is possible to give a single-fold Diophantine definition for every c.e.\ set of integers (see Theorem \ref{thm:together}).

Using these results on single-fold definability of $\Z$ and c.e. sets of integers, following results of   Jan Denef from \cite{Den0} and Karim Zahidi from \cite{Zahidi2000}, we show in Section \ref{polyce} that all computably enumerable subsets of a polynomial ring over a ring of integers of a totally real number field are finite-fold existentially definable.  As far as we know, this is the first example of this kind. (See Theorem \ref{Zahidi-finitefold}.)

In Section \ref{sec:gence} we generalize results of Jeroen Demeyer from \cite{Demeyer10} to show that all c.e.\ subsets of rings of integral functions over number fields are Diophantine (see Theorem \ref{thm:extce}).   In order to do so, we needed to generalize the earlier treatments (given in the papers \cite{MB3} of Laurent Moret-Bailly and \cite{Eispadic} of Kirsten Eisentraeger) of definability of integrality at a degree one valuation over a function field of characteristic zero where the constant field is a number field.
Both of those papers in turn extend results of H. K. Kim and Fred Roush from \cite{K-R3} where the two authors give a Diophantine definition of integrality at a valuation of degree 1 over a rational function field with a constant field embeddable into a $p$-adic field.  (Such constant fields include all number fields.)    The papers of Moret-Bailly and Eisentraeger are primarily concerned with extending results pertaining to Hilbert's Tenth Problem and so they extend the results of Kim and Roush just enough for their arguments to go through, by showing the following for a function field $K$ over a field of constants $k$ as described above:  if $T$ is a non-constant element of $K$ and the pole $\qq$ of $T$ splits completely into distinct primes in the extension $K/k(T)$, then there exists a Diophantine subset of $K$ such that all rational functions in that subset are integral at $\qq$.  

In contrast, our proof required a Diophantine definition of the valuation ring of a single factor of $\qq$ in $K$. In order to obtain such  a definition we reworked the original construction of Kim and Roush.  In this paper we show that the valuation ring of ``almost'' any prime of a function field $K$ of characteristic 0 is existentially definable over a function field with a constant field algebraic over $\Q$ and embeddable into a finite extension of  $\Q_p$ for $p \ne 2$ or $\R$.  The "almost" part applies only to the case where we have to use the fact that the constant  field under consideration is embeddable into $\R$.  If the constant field is embeddable into a finite extension of $\Q_p$, with $p \ne 2$, then we can give an existential definition (with a parameter, of course) of {\it any} valuation ring.  We also should note here that the class of constant fields described above contains an infinite subset of non-finitely generated fields. (See Section \ref{section:order}.)


\section{Number Fields, Function Fields and Rings}
\label{sec:function}
Throughout this paper, by a function field $K$  we will mean a finite extension of a rational function  field $k(T)$, where $T$ is transcendental over a field $k$ of characteristic 0.  By the constant field of $K$ we will mean the algebraic closure of $k$ in $K$.  (In our case we will often have a situation where the algebraic closure of $k$ in $K$ is equal to $k$ by construction.)  

  By a prime $\pp$ of $K$, we will mean the maximal  ideal of the valuation ring $R_v$ corresponding to one of the valuations $v$ defined on $K$. (See \cite{C} for more details on valuations of function fields.)  Given $f \in R_v, f \ne 0$, we  define $\ord_{\pp}f$ to be the largest non-negative integer $n$ such that $f \in \pp^n$.  If $f \not \in R_v$, then $\frac{1}{f} \in R_v$, and we define $\ord_{\pp}f=-\ord_{\pp}\frac{1}{f}$. We define the order of 0 to be infinity.

 For the field of constants we will most often select some algebraic extension of $\Q$.  When such an extension is finite, the field is called a number field.  One can also consider all the possible embeddings of $k$ into $\widetilde \Q$, the algebraic closure of $\Q$ inside $\mathbb C$.  If all the embeddings are contained in $\R$, then the field is called totally real.  If a field is algebraic over $\Q$ and has an embedding into $\R$, then we will call it formally real.

\section{Single-fold Diophantine Representations of C.E.\ Sets of Integers over Rings of $\calS$-integers of Function Fields of Characteristic 0}
\label{intsingle}
In this section we describe a finite-fold Diophantine definition of c.e.\ sets of integers over rings of integral functions. Before we do that, we have to reconsider certain old methods of defining sets to make sure they produce single-fold definitions.  We start with the issue of intersection of Diophantine sets.
\subsection{Single-Fold and Finite-Fold Definition of ``And''}
As long as we consider rings whose fraction fields are not algebraically closed, we can continue to use the ``old'' method of combining several equations into a single one without introducing extra solutions, as in Lemma 1.2.3 of \cite{Sh34}.  More specifically we have the following proposition.
\begin{proposition}
Let $R$ be an integral domain such that its fraction field is not algebraically closed.   Let $k \in \Z_{>0}$ and  $A, B, C \subset R^k$ be Diophantine subsets of $R^k$ such that the sets $B$ and $C$ have single (finite) fold definitions, and $A=B\cap C$.  Then $A$ has a single (finite) fold definition.  

More specifically, if we let $h(T) = a_0 + a_1T + \ldots +T^n$ be a polynomial without roots in the fraction field of $R$, let $\bar x=(x_1,\ldots, x_k), \bar z =(z_1,\ldots, z_m)$,  let $f(\bar x, \bar z)$ be a single (finite) fold definition of $B$, and  let $g(\bar x, \bar z)$ be a single (finite) fold definition of $C$, then 
\[
h(\bar x, \bar z)=a_0f(\bar x, \bar z)^n + a_1f(\bar x, \bar z)^{n-1}g(\bar x, \bar z) + \ldots +g(\bar x, \bar z)^n
\]
 is a single (finite) fold definition of $A$.
\end{proposition}

The proof of this proposition is the same as for Lemma 1.2.3 of \cite{Sh34}.

\subsection{Pell Equations over Rings of Functions of Characteristic 0}
Next we take a look at the old workhorse of Diophantine definitions: the Pell equation.  It turns out that in the context of defining integers over rings of functions this equation produces ``naturally'' single-fold definitions.
\begin{lemma}(Essentially Lemma 2.1 of \cite{Den5}, or Lemma 2.2 of  \cite{Sh90} ) 
\label{Pell1}
Let $Z$ be an integral domain of characteristic not equal to 2. Let $f, g, v \in Z[x],  v, x$ transcendental over $Z$.  Let $(f_n(v),g_n(v)) \in Z[x]$ be such that  $ f_n(v) - (v^2 - 1)^{1/2}g_n(v) = (v - (v^2 - 1)^{1/2})^n.$  (In \cite{Sh90} there is a typographical error in this equation: the ``square'' is misplaced on the right-hand side.)  In this case
\begin{enumerate}
\item $\deg{f_n} = n\cdot \deg{v}, \deg{g_n} = ( n - 1)\cdot\deg{v}$,
\item $ \ell \mbox{ dividing }n  \mbox{ is equivalent to }g_{\ell} \mbox{ dividing } g_n$,
\item The  pairs $(\pm f_n,\pm g_n)$ are all the solutions to $f^2 - (v^2 - 1)g^2 = 1$ in $Z[x]$
\end{enumerate}
\end{lemma}
Below we use the following notation.
\begin{notation} \textup{}
\label{notoriginal}
\begin{itemize}
\item  Let $K$ denote a function field of characteristic 0 over the constant field $k$.
\item Let $\calS=\{\pp_1, \ldots, \pp_s\}$ be a finite non-empty set of primes of $K$.
\item Let $O_{K,\calS}=\{x  \in K~|~(\forall \pp \not \in \calS)~\ord_{\pp}x \geq 0 \}$ be the ring of $\calS$-integers of $K$.
\item For $d \in O_{K,\calS}$ such that $d$ is not a square in $K$, let 
\[
H_{K,d,\calS}= \{x-d^{1/2}y~|~x,y \in O_{K,\calS}~\&~x^2-dy^2=1\}.
\]
\end{itemize}
\end{notation}
The next proposition follows from Lemmas 2.3,2.4, 2.5, 3.1--3.3 of \cite{Sh92}.
\begin{proposition}
\label{prop:exist}
There exists $a \in O_{K,\calS}$ satisfying the following conditions.
\begin{itemize}
\item $\ord_{\pp_1}a<0$,
\item $\ord_{\pp_i}(a-1)>0$ and $\ord_{\pp_i}(a-1)$ is odd for $i=2,\ldots, s$.
\item In $K(\sqrt{a^2-1})$, the prime $\pp_1$ splits into the product $\qq_{\infty}\qq$.
\item The element $T=a-\sqrt{a^2-1}$ of $K(\sqrt{a^2-1})$ has a zero at $\qq$,  a pole at $\qq_{\infty}$, and no other zeros or poles. 
\item For $d=a^2-1$, we have that $H_{K,d,\calS}$ is a cyclic group generated by $a-d^{1/2}$ modulo $\{\pm 1\}$, and  $a-1$ and $a-\sqrt{a^2-1} - b$ (for any $b \in \Z_{\not =0}$) are not  units of $O_{K,\calS}[\sqrt{a^2-1}]$.
\item Let $f_n, g_n$ be as in Lemma \ref{Pell1}. Then $g_{-n} = -g_n$, $f_{-n}=f_n$, and 
\[
g_n \equiv n \mod (a - 1)
\]
in the ring $\Z[a]$.  (Note that $f_n, g_n \in \Z[a]$.)
\end{itemize}
\end{proposition}
\begin{remark}
\label{rem:signofepsilon}
Let $u, w \in O_{K,\calS}$, let $a$ be as in Proposition \ref{prop:exist}.  Finally assume $u^2- (a^2-1)w^2=1$.  Then what can we say about $u-\sqrt{a^2-1}w$?  By Proposition \ref{prop:exist}, there exists $n \in \Z$ such that $u=\pm f_n, w=\pm g_n$, where 
\[
(a-\sqrt{a^2-1})^n=\varepsilon^n=f_n -\sqrt{a^2-1}g_n.  
\]
Thus, we have four possibilities: 
\[
u-\sqrt{a^2-1}w=f_n-\sqrt{a^2-1}g_n=\varepsilon^n, n \in \Z_{\geq 0}, 
\]
\[
u-\sqrt{a^2-1}w=f_n+\sqrt{a^2-1}g_n=\varepsilon^n, n \in \Z_{\leq 0}, 
\]
\[
u-\sqrt{a^2-1}w=-f_n-\sqrt{a^2-1}g_n=-\varepsilon^n, n \in \Z_{\leq 0},
\]
\[
 u-\sqrt{a^2-1}w=-f_n+\sqrt{a^2-1}g_n=-\varepsilon^n, n \in \Z_{\geq 0}. 
\]
Alternatively, 
\[
u-\sqrt{a^2-1}w=\pm \varepsilon^n, n \in \Z. 
\]

\end{remark}
\begin{lemma}(Essentially Lemma 3.4 of \cite{Sh92}.)
\label{le:containslocal}
   Let  $R$ be any subring  of  $O_{K,\calS}$ containing a local subring of $\Q$. (In particular, $R$ can be equal to $O_{K,\calS}$.) Then there exists a subset $C$ of $R$ that contains only constants, includes $\Z$, and is single-fold Diophantine  over $R$.
\end{lemma}

\begin{proof}
    We remind the reader that the $\calS$ contains $s$ primes. Let $\pi$ be the product of all non-invertible rational primes (or 1, if $R$ contains $\Q$), and let $C \subset R$ be the set  defined by the 
following equations over $R$:
\begin{equation}
\label{invert}
\left \{
\begin{array}{c}
j_1(\pi x^2 +1)=1,     \\
\ldots         \\                     
j_{s+1}(\pi x^2+(s+1)\pi +1)=1
\end{array}
\right.
\end{equation}

We claim that  System \eqref{invert} has solutions in $R$ only if $x$ is a constant, while
conversely,  if $x\in \Z$, then  these equations have solutions in $R$. Indeed, if $x$ is not a constant, neither are 
$x^2 + \pi +1, \ldots, x^2+ (s+1)\pi +1$.   Therefore, since they are invertible  in  $O_{K,\calS}$, they  all must  have zeros at  
valuations  of  $\calS$. However, these $s+1$ elements do not share any zeros, and
there are $s$ valuations  in  $\calS$ and  $s + 1$ elements  under  consideration.  This
implies that two of them must share the same zero, which is impossible since their differences are 
constant. The converse is obvious: if $x \in \Z$, then $\pi x^2+\pi r +1$ is invertible for each $r \in \Z$.  Please note that given $x \in O_{K,\calS}$, if System \eqref{invert} has solutions, then these solutions are unique.
\end{proof}

\begin{notation}
Let $J(x)$ denote the system of equations \eqref{invert}.  
 \end{notation}
 We will use this system to give a single-fold Diophantine definition of $\Z$ over $O_{K,\calS}$.
\begin{theorem}
\label{thm:singlefoldZ}
$\Z$ has a single-fold Diophantine definition over $O_{K,\calS}$.
\end{theorem}
\begin{proof}
There are several ways to state this proof.  We choose the way that we will later use to produce a single-fold definition of exponentiation for $\Z$.  Let $a \in O_{K,\calS}$ be  as in Proposition \ref{prop:exist} and consider the following equations and conditions:
\begin{equation}
\label{eq:one}
u^2-(a^2-1)w^2=1;
\end{equation}

\begin{equation}
\label{eq:ratio}
c \equiv w \bmod (a-1)
\end{equation}
 in  $O_{K,\calS}$;
\begin{equation}
\label{eq:last}
J(c).
\end{equation} 
Supposed now that Equations \eqref{eq:one}--\eqref{eq:last} are satisfied with variables ranging over $O_{K,\calS}$.   Then, by Proposition \ref{prop:exist} and Lemma \ref{Pell1},  $w =w_n \equiv  n \bmod (a-1)$ for some $n \in \Z$.   Thus, $c \equiv  n \mod (a-1)$ in $O_{K,\calS}$.  Since $a-1$ is not a unit of $O_{K,\calS}$, we have that  $c-n$ is a constant with a zero at some valuation of $K$.  Hence $c= n$.  

Conversely, given $c = n \in \Z$, set $w=w_n$ and observe that all the equations are satisfied.  Note also, that this is the only solution to the equations.
\end{proof}
\begin{notation}
Let $U(c,u,w)$ denote the system of equations \eqref{eq:one}--\eqref{eq:last}.
\end{notation}
We now give a single-fold Diophantine definition of exponentiation.
\begin{theorem}
\label{thm:exp}
The following set has a single-fold definition over $O_{K,\calS}$: $$\{(b,c,d)~|~b, c, d \in \Z_{\ne 0}, d>0, c=b^d\}.$$
\end{theorem}
\begin{proof}
Consider the following system of equations:
\begin{equation}
\label{eq:not0}
b \ne 0, c \ne 0, d \ne 0,
\end{equation}
\begin{equation}
\label{eq:in Z}
U(c,u_c,w_c).
\end{equation}
\begin{equation}
\label{eq2:in Z}
U(b,u_b,w_b),
\end{equation}
\begin{equation}
\label{eq3:in Z}
U(d,u_d,w_d),
\end{equation}
\begin{equation}
\label{eq:equiv1}
\exists n \in \Z, x, y\in O_{K,\calS}[\sqrt{a^2-1}]: \pm \varepsilon^n - c=(\varepsilon - b)x \land d -\frac{\pm \varepsilon^n-1}{\varepsilon-1}=y(\varepsilon-1).
\end{equation}
We start with noting that \eqref{eq:equiv1} implies that $(\varepsilon -1)$ divides $\pm \varepsilon^n -1$ in $O_{K,\calS}$.  Note that by Proposition \ref{Pell1}, we also have that $\varepsilon -1$ is not a unit of $O_{K,\calS}$.  If we choose the ``minus'' option, then we have that $\varepsilon -1$ divides $\varepsilon^n +1$.  Since $\varepsilon -1$ divides $\varepsilon^n-1$, it follows that $\varepsilon-1$ divides $2$, and thus is a unit of $O_{K,\calS}$.  Hence ``minus'' option cannot occur.  Consequently, \eqref{eq:equiv1} can be rewritten as:
\begin{equation}
\label{eq:equiv2}
\exists n \in \Z, x, y\in O_{K,\calS}[\sqrt{a^2-1}]:  (\varepsilon^n - c=(\varepsilon - b)x) \land (d -\frac{\varepsilon^n-1}{\varepsilon-1}=y(\varepsilon-1)).
\end{equation}

From \eqref{eq:equiv2} we deduce that $\varepsilon^n - c \equiv 0 \bmod (\varepsilon -b)$ in $O_{K,\calS}$.  At the same time $\varepsilon^n \equiv   b^n  \bmod (\varepsilon -b)$ in $O_{K,\calS}$.  Therefore, $c \equiv b^n \bmod (\varepsilon-b)$.  By Proposition \ref{prop:exist} we conclude that $(\varepsilon -b)$ is not a unit, and therefore $ b^n-c$ has a zero at a valuation of $K$.  Since $b, c$ are integers, we must infer that $b^n=c$, and $n\geq 0$.  At the same time, also from \eqref{eq:equiv2}, we have that $d \equiv n \bmod (\varepsilon-1)$ in $O_{K,\calS}$.  By the same argument as above we conclude that  $d=n>0$.

Conversely, assuming $b, c, d \in \Z, bcd \not=0, d>0, c=b^d$, it is easy to see that \eqref{eq:equiv1} can be satisfied with only one choice for the sign in front of $\varepsilon^d$.

We now rewrite \eqref{eq:equiv1} over $O_{K,\calS}$:
\begin{equation}
\label{rewritten}
\left \{
\begin{array}{r}
u^2-(a^2-1)w^2=1 \, (\mbox{in other words, } \pm\varepsilon^n =  u - \sqrt{a^2-1}w)\\
u-\sqrt{a^2-1}w - c=(a -\sqrt{a^2-1} -b)(x_1-x_2\sqrt{a^2-1} ) \\
 (\mbox{in other words, } \varepsilon^n -c=(\varepsilon-b)x )\\
d(a -\sqrt{a^2-1} -1) - (u-\sqrt{a^2-1}w-1)=(y_1-\sqrt{a^2-1}y_2)(a -\sqrt{a^2-1} -1)^2 \\
  (\mbox{in other words, }d(\varepsilon - 1)-(\pm \varepsilon^n-1)=y(\varepsilon-1)^2)
\end{array}
\right .
\end{equation}
Thus System \eqref{rewritten}, with all the variables ranging over $O_{K,\calS}$, is equivalent to Conjunction \eqref{eq:equiv1}.
\end{proof}
\begin{notation}
\label{not:exp}
For future reference we will denote the equations \eqref{eq:not0}--\eqref{eq3:in Z} together with \eqref{rewritten} by 
\[
G(a,b,c,d,u, w, x_1,x_2,y_1,y_2).
\]
\end{notation}
\begin{corollary}[Single fold definition of positive integers over $O_{K,\calS}$]
\label{cor:plus}
Let $a$ be as in Proposition \ref{prop:exist}, and let 
\[
\text{Plus}=\{d \in O_{K,\calS}| \exists c,u,w,x_1, x_2, y_1, y_2 \in O_{K,\calS}: G(a,2,c,d,u,w,x_1,x_2,y_1,y_2) \}.
\]
 Then Plus$=\Z_{>0}$, and this Diophantine definition is single-fold.
\end{corollary}
\begin{proof}
Given Theorem \ref{thm:exp}, the only point that needs a proof is the single-fold property of the definition.  Given a $d>0$, to satisfy the system, we must have that 
\[
c=2^d, u=f_d, w=g_d,
\]
\[
 x_1-x_2\sqrt{a^2-1} =\frac{u-\sqrt{a^2-1}w - 2^d}{a -\sqrt{a^2-1} -2},
 \]
 \[
  y_1-\sqrt{a^2-1}y_2=\frac{d(a -\sqrt{a^2-1} -1) - (u-\sqrt{a^2-1}w-1)}{(a -\sqrt{a^2-1} -1)^2}. 
 \]
  Thus the values of all variables are uniquely determined by $d$.

\end{proof}
We also have another corollary to be used in Section \ref{sec:gence}.
\begin{corollary}
\label{s,Ts}
The set $\{(s,u_s, w_s)|s \in \Z_{>0}\}$ is single-fold Diophantine over $O_{K,\calS}$.
\end{corollary} 
\begin{proof}
Consider the set 
\[
\{(s,u,w) \in O_{K,\calS}^3:  \exists c,x_1, x_2, y_1, y_2 \in O_{K,\calS}: G(a,2,c,s,u,w,x_1,x_2,y_1,y_2) \}.
\]
By the same argument as in the proof of Corollary \ref{cor:plus}, the set consists of all the triple in the required form, and given such a triple, the values of all the other variables are determined uniquely.
 \end{proof}
\bigskip
Combining Theorem \ref{thm:singlefoldZ}, Corollary \ref{cor:plus} with a result of Matiyasevich from \cite{Mat77} we now have the following theorem.
\begin{theorem}
\label{thm:finite-foldceZ}
Every c.e.\ set of integers has a single-fold Diophantine definition over $O_{K,\calS}$.
\end{theorem}

\section{Single-Fold Diophantine Representations of C.E.\ Sets of Rational Integers over Polynomial Rings of Characteristic 0.}
 In this section we prove the analog of Theorem \ref{thm:finite-foldceZ}, but for polynomial rings over arbitrary commutative rings with unity of characteristic 0.  If the ring of constants contains $\Q$, then the proof of Theorem \ref{thm:finite-foldceZ} can be used verbatim.  So  the only case that we need to consider is the situation where the ring of constants does not contain  $\Q$.  If the constant ring does not contain $\Q$, it can contain infinitely many non-invertible primes, and therefore we cannot use the definition of a constant set containing all integers from Lemma \ref{le:containslocal}.  For exactly same reason, we cannot use multiplicative inverses to define the set of non-zero elements.  Thus, we will have to modify some parts of the proof of Theorem \ref{thm:finite-foldceZ}.\\
 
 First we need the following basic fact.
\begin{lemma}
\label{le:inverse}
If $a, b$ are non-zero relatively prime integers, then $\frac{1}{b} \in \Z[\frac{a}{b}]$.
\end{lemma}

\begin{proof}
Since $(a,b)=1$ we have that for some $x_1, x_2 \in \Z$ it is the case that $ax_1 +bx_2=1$.   Thus $\displaystyle \frac{1}{b}=\frac{ax_1+bx_2}{b}=x_1\frac{a}{b}+ x_2\in \Z[\frac{a}{b}]$.\\
\end{proof}
Next we deal with the question of saying that an element is not 0.  

\begin{lemma}
\label{le:non-zero}
Let $R$ be a ring of characteristic $0$ such that a rational prime $p$ does not have an inverse in the ring.  In this case, there exists a set $A=A_p$  such that $0 \not \in A$,  $p\Z +1 \subset A$, and if $px+1 \in A\cap \Z$, then $x \in \Z$.
\end{lemma}
\begin{proof}
Let $A =\{px+1~|~x \in R\}\subset R$. Then $0 \not \in A$.  Indeed, if $0 \in A$ then $\frac{1}{p} \in R$, and we have a contradiction. Suppose now that for some $x \in R$ we have that $px+1 \in \Z$. We claim that $x \in \Z$. 
Observe that  if $px+1 \in \Z$ then $px=z \in \Z$ and $\frac{z}{p}=x \in R$.  If $x \not \in \Z$, then $(p,z)=1$, and $p$ has an inverse in $R$ by Lemma \ref{le:inverse}, in contradiction of our assumptions.  Finally, clearly $p\Z+1 \in A_p$.
\end{proof}

\begin{theorem}(Similar to Theorem 5.1 of \cite{Sh90})
\label{powersofT}
 If $Z$ is an integral domain of characteristic 0 and $x$ is transcendental
over $Z$, then $\Z$ is single-fold Diophantine over $R=Z[x]$.
\end{theorem}
\begin{proof}
As we explained above, without loss of generality, we can assume that $\Q \not \subset R$, and therefore, by Lemma \ref{le:inverse}, $R$ contains at least one non-inverted prime.
 Consider the following set of equations,
\begin{equation}
\label{eq:0}
(f_i-\sqrt{(a^2x^2-1)}g_i)=(ax-\sqrt{(a^2x^2-1)})^i, i=2, 3
\end{equation}
\begin{equation}
\label{eq:1}
f^2  -(a^2 x^ 2 -1)g ^2 =1,
\end{equation}
\begin{equation}
\label{eq:2}
f-\sqrt{a^2x^2-1}g-1=(f_3\sqrt{a^2x^2-1}g_3-1)(z_1-\sqrt{a^2x^2-1}z_2)
\end{equation}
\begin{equation}
\label{eq:3}
t | g_3g_2, 
\end{equation}
\begin{equation}
\label{eq:4}
t \equiv g \mod g_3^2,
\end{equation}
\begin{equation}
\label{eq:5}
ax | f,
\end{equation}
\begin{equation}
\label{eq:6}
a = t/g_3,
\end{equation}

We show that these equations can be satisfied with some values of variables 
\[
a \not = 0, f, g, f_2, g_2, f_3, g_3, t, z_1, z_2  \in Z[x]
\]
 only if we choose $a$ to be an odd integer. First of all, we note that for any choice of $a\in Z[x]\setminus{0}$, we have that $a^2x^2-1 \not \in \Z$.  Indeed, if $a \in Z[x]$ and $a \ne 0$, then $\deg a$ as a polynomial in $x$ is bigger or equal to 0.  Therefore, the degree  of $a^2x^2-1$ is bigger or equal to 2.  By Lemma \ref{Pell1} we have from \eqref{eq:1} that $ f =\pm f_n,  g = \pm g_n$ for some $n \in \Z_{\geq 0}$.  Alternatively, $g=g_m, m \in \Z$.  
 Let $\varepsilon=ax-\sqrt{a^2x^2-1}$.  Then $f-\sqrt{a^2x^2-1}g=\pm \varepsilon^m$ for some $m \in \Z$.  (See Remark \ref{rem:signofepsilon} for the discussion of signs.)  Next  by Proposition \ref{prop:exist}, we deduce that $\varepsilon^3-1$ is not a unit.  So, from \eqref{eq:2}, as in the proof of Theorem \ref{thm:exp}, we conclude that $f-\sqrt{a^2x^2-1}g= \varepsilon^m$,
$m = 3r, r \in \Z_{\ne 0}$.  Further,  from \eqref{eq:5} we obtain that $f_1 | f_m$, implying that $m$ is odd. (From the binomial expansion, it  is easy to see that $f_1$ divides $f_m$ in the polynomial ring only if $m$ is odd.)  Hence,  $r$ is odd.
From Lemma \ref{Pell1} we also have that
\[
g_{3r}  = \pm  \sum_{|r|-i \mbox{ \tiny odd}}\binom{|r|}{i}
f_3^i((ax)^2 - 1)^{(|r|-i- 1)/2}g_3^{|r|-i},
\]
where ``$-$'' corresponds to $r<0$. Thus $g_{3r} \equiv rf_3^{|r|-1}g_3 \bmod g_3^2$.
Additionally, we have that   $f_3^2 \equiv 1 \bmod g_3^2$. Since $|r| - 1$ is even, we
now deduce $g_{3r} \equiv  rg_3 \bmod g_3^2$.  Thus, we conclude using \eqref{eq:4} that $ t \equiv rg_3 \bmod g_3^2$ or equivalently 
\begin{equation}
\label{eq:divide}
g_3^2 | (t - rg_3).
\end{equation}
From \eqref{eq:3} we have $t | g_3g_2$ so that $\deg{t} < 2 \deg{g_3}$, and $\deg{t -rg_3} <\deg{g_3^2}$. Therefore \eqref{eq:divide} implies that $t - rg_3=0$, $a= t/g_3=  r$, that is, $a$ is an odd
integer.

Conversely, suppose $r$ is an odd integer and let $a = r$. Then $t=ag_3=rg_3$.  To satisfy \eqref{eq:1} and \eqref{eq:2} we need to set $(f, g) = (f_{m}(ax), g_{m}(ax))$, where $m \ne 0, m \in \Z$.  Further, \eqref{eq:2}  requires that $m \equiv 0 \bmod 3$.  To satisfy \eqref{eq:4}, we need to arrange for $t\equiv g \bmod g_3^2$, or in other words, we need $ag_3 -g_m \equiv 0 \bmod g_3^2$ to be satisfied.  As before, \eqref{eq:5} implies $m$ is an odd number.  So we have to set $m=3r'$, where $r'$ is odd.  Thus, again as above we have that $g_m \equiv  r'g_3 \bmod g_3^2$.   Therefore, we have to choose $r' \equiv r \bmod g_3$.  But since both $r', r \in \Z$, and $g_3 \not \in Z$, the only way to satisfy the equivalence is to set $r'=r$.
Now \eqref{eq:0}--\eqref{eq:2}, \eqref{eq:4} and \eqref{eq:5} are satisfied.
Since $g_2(rx) = 2rx$, and we set $t = rg_3(rx)$, we can conclude that $t \Big{|} g_3(rx)g_2(rx)$, and \eqref{eq:3} and \eqref{eq:6} are satisfied.  Observe, that given an odd integer $a$, the remaining variables have to take the values described above.

We now show how to state the assumption that $a\ne 0$.   Let $p$ be a rational prime without a multiplicative inverse in $R$.   We replace the condition $a \ne 0$ by $a=2ps+1, s \in R$.  By Lemma \ref{le:non-zero}, the added equation will imply that $a \ne 0$.

Now, if Equations \eqref{eq:0}--\eqref{eq:6} together with the new equation $a=2ps+1$ are satisfied, by Lemma  \ref{le:non-zero}, we conclude that $a=2ps+1$  is an odd integer, i. e. $a=2u+1$ for some $u \in \Z$.  Therefore  $2ps+1 =2u+1$ or $sp =u \in \Z$.  Since $ps \in \Z$, the only prime that can divide the denominator of $s$ is $p$.  But by Lemma \ref{le:inverse}, we have that $p$ cannot appear in a reduced denominator of an element in $R$.  Therefore, we conclude that $s \in \Z$.  Hence, Equations \eqref{eq:0}--\eqref{eq:6} together with the new equation $a=2ps+1$ are satisfiable over $R$ if and only if $s \in \Z$.

\end{proof}
\begin{notation}
\label{not:inZ}
We denote Equations \eqref{eq:0} -- \eqref{eq:6}, together with the equation $a=2ps+1$ by $F(a,x, f, g, f_2, g_2, f_3, g_3, t, p, s)$.  Thus, 
\[
\Z=\{s \in Z[x]|F(a,x, f, g, f_2, g_2, f_3, g_3, t, p, s)\}.
\]
In this formula $p$ is a fixed parameter corresponding to a prime not inverted in $R$.
\end{notation}
In this section, as in the section concerning rings of $\calS$-integers, we will need to know that for any $a \in R, b \in \Z_{\ne 0}$ we have that $\varepsilon^n -b=(ax-\sqrt{a^2x^2-1})^n-b$ is not a unit of $R[\sqrt{a^2x^2-1}]$.   There is a similar statement in Proposition \ref{prop:exist}.  The difference between that statement and the statement below is that over $O_{K,\calS}$ we fixed $a$, while here $a$ ranges over $R$.
\begin{lemma}
\label{le:nonconstant}
Let $a \in Z[x],  b, n \in \Z, abn \ne 0$.  Then we have that $\varepsilon^n-b$ is not a unit in $R[x, \sqrt{a^2x^2-1}]$. 
\end{lemma}
\begin{proof}
 Let $K$ be the fraction field of $R$.    Let $f_n(ax)-\sqrt{a^2x^2-1}g_n(ax)=\varepsilon^n(ax)=\varepsilon^n$.  The monic irreducible polynomial of $\varepsilon^n$ over $Z[x]$ is of the form $X^n-2f_nX+1$.  Therefore, for any $b \in \Z$, we have that ${\mathbf N}_{K(\sqrt{a^2x^2-1})/K}(b-\varepsilon^n)= b^n-2f_nb+1$.  Here we note that the only units of $Z[x]$ are some elements of $Z$.  So, if $b^n-2f_nb+1 \not \in Z$, then we conclude that $b^n-2f_nb+1$ is not a unit.   Since $f_n \not \in Z, b \ne 0$, we conclude that $b^n-2f_nb+1$ is not a unit in $Z[x]$, and, therefore, $\varepsilon - b$ is not a unit in 
 $Z[x,\sqrt{a^2x^2-1}]$. 

\end{proof}
Now that we have a single fold Diophantine definition of integers, we can produce a single-fold definition of non-zero integers, and then positive integers and exponentiation.
\begin{lemma}
If $Z$ is an integral domain of characteristic 0 and $x$ is transcendental
over $Z$, then $\Z_{\ne 0}$ is single-fold Diophantine over $R=Z[x]$.
\end{lemma}
\begin{proof}
Let $s \in R$ be given and consider the following sequence of equations:
\begin{equation}
\label{eq:nonzero1}
F(a,x, f, g, f_2, g_2, f_3, g_3, t, p, s),
\end{equation}
\begin{equation}
\label{eq:nonzero2}
u^2-(x^2-1)w^2=1,
\end{equation}
\begin{equation}
\label{eq:nonzero3}
s - \frac{u-w\sqrt{x^2-1}-1}{x-\sqrt{x^2-1} -1}=(z_1-z_2\sqrt{x^2-1})(x-\sqrt{x^2-1}).
\end{equation}
From \eqref{eq:nonzero1} via Notation \ref{not:inZ}, we conclude that $s \in \Z$.  If we set $\varepsilon=\varepsilon(x)=(x-\sqrt{x^2-1})$, then by Lemma \ref{Pell1} and Remark \ref{rem:signofepsilon} we deduce from \eqref{eq:nonzero3} that  that $u-\sqrt{x^2-1}v=\pm \varepsilon^m, m \in \Z$.    By Lemma \ref{le:nonconstant} we also know that $\varepsilon-1$ is not a unit of $R$.  So, as in the proof of Theorem \ref{thm:exp} again, $u-\sqrt{x^2-1}v= \varepsilon^m$, and $m \ne 0$.  Finally, we have that $s \equiv m \bmod (\varepsilon-1)$ in $R$, implying $s=m \ne 0$.

To see that this definition is single-fold, let $s \in R$.  Then  by Theorem \ref{powersofT}, there is a uniques set of values that can be taken by variables  $a,x, f, g, f_2, g_2, f_3, g_3, t \in R$.  Finally, given $s \in \Z$, Equation \eqref{eq:nonzero2} forces $u-w\sqrt{x^2-1}=(x-\sqrt{x^2-1})^s$.  Thus, $u,w$ are determined uniquely.
\end{proof}
\bigskip
The next theorem will show existence of  single-fold  definitions for all c.e. subsets of rational integers in polynomial rings. The proofs will proceed via Diophantine definitions of  exponentiation with arguments similar to the ones used in the proof  of Theorem \ref{thm:exp}.
\begin{notation}
Let $\hat U(s, \ldots)$ denote Equations \eqref{eq:nonzero1}-\eqref{eq:nonzero3}.  So that over $R$
\[
\Z_{\ne 0}=\{s|\exists ....\hat U(s,\ldots)\},
\]
 and this definition is single-fold.
\end{notation}

\begin{theorem}
\label{thm:withoutQandwithQ}
Let $Z$ be an integral domain of characteristic 0, and assume $x$ is transcendental over the fraction field of $Z$ and $\Q \not \subset R$.  Then every c.e. set of rational integers has a single-fold Diophantine definition over $R=Z[x]$.
\end{theorem}
\begin{proof}
We can now proceed as in the case of the rings of integral functions.  
Consider the following system of equations:
\begin{equation}
\label{eq:in Z1}
\hat U(c,\ldots).
\end{equation}
\begin{equation}
\label{eq2:in Z1}
\hat U(b,\ldots),
\end{equation}
\begin{equation}
\label{eq3:in Z1}
\hat U(d,\ldots),
\end{equation}
\begin{equation}
\label{eq:equiv11}
\exists n \in \Z, x, y\in R: [\sqrt{a^2-1}]: \pm \varepsilon^n - c=(\varepsilon - b)x \land d -\frac{\pm \varepsilon^n-1}{\varepsilon-1}=y(\varepsilon-1).
\end{equation}
By the same argument, as in the case of integral functions, this system gives a single-fold definition of the set 
\[
\hat G=\{(b,c,d) \in \Z_{\ne 0}^3, d>0, c=b^d\}.
\]
Finally, if we set $b=2$, we get the set $\hat G=\{(2,c,d)\in \Z^3_{>0}|c=2^d\}$.  Using Matiyasevich's result, we now conclude that the assertion of the theorem holds.
\end{proof}
Finally, putting the result above together with our observation about the case when $\Q \subset R$, we obtain the following result.
\begin{theorem}
\label{thm:together}
Let $Z$ be an integral domain of characteristic 0, and assume $x$ is transcendental over the fraction field of $Z$.  Then every c.e. set of rational integers has a single-fold Diophantine definition over $R=Z[x]$.
\end{theorem}

\section{Finite-fold Diophantine Definition of C.E. Sets of Polynomial Rings over Totally Real Fields of Constants}
\label{polyce}
So far we have produced single-fold definitions of certain c.e.\ subsets of a ring.  We now construct our first examples of rings where all c.e.\ sets have finite-fold definitions.  To do this we combine the arguments above with the proof of Zahidi from \cite{Zahidi2000} showing that over a polynomial ring with coefficients in a ring of integers of a totally real number field, all c.e.\ sets were Diophantine.  Zahidi's result was in turn an extension of a result of Denef from \cite{Den0} in which the coefficients of the polynomial ring came from $\Z$.  

Any discussion of c.e.\ sets of a polynomial ring and a ring of integral functions to be discussed later, must of course involve some discussion of indexing of the ring.  In other words we will need a bijection from a ring into the positive integers such that given a ``usual'' presentation of a polynomial (or an integral function in the future) we can effectively compute the image of this polynomial (or this integral function), and conversely, given a positive integer, we can determine what polynomial (or integral function) was mapped to it.  For a discussion of an effective indexing map in the case of a rational function field we refer the reader to the paper of Zahidi.  A discussion of indexing for function fields can be found in \cite{Sh34}.    In this paper we will assume that such an indexing is given and, following Zahidi, will denote it by $\theta$ going from positive integers to polynomials.  Below we describe the rest of our notation and assumptions.
\begin{notationassumptions}
\hfill
\begin{itemize}
\item Let $k$ be a totally real number field.
\item Let $O_k$ be the ring of integers of $k$.
\item Let $\alpha_1,\ldots, \alpha_r$ be an integral basis of $O_k$ over $\Z$.
\item Let $\theta: \Z_{>0} \longrightarrow O_k[T]$ be the effective bijection discussed above. 
\item Define  $P_n(X):=\theta(n)$.
\item Let $(U_n(X),W_n(X)) \in O_k[X]$ be such that  
\[
 U_n(X) - (X^2 - 1)^{1/2}W_n(X) = (X- (X^2 - 1)^{1/2})^n.
 \]
\end{itemize}
\end{notationassumptions}
As we indicated before, our intention is to follow the plan laid out by Zahidi and Denef, just making sure that all the definitions in that plan are finite-fold.  This plan entails showing (a) that all c.e.\ subsets of $\Z$ are (finite-fold) Diophantine over the polynomial ring in question and (b) that the indexing is (finite-fold) Diophantine or, in other words,  the set 
\[
\{(n, P_n(X))| n \in \Z_{>0}\}
\]
 is (finite-fold) Diophantine over $O_k[X]$.  Zahidi provides a brief argument in his paper that we apply to our situation, given that we have a finite-fold way of combining equations, to see that (a) and (b) imply the following theorem.
\begin{theorem}
\label{Zahidi-finitefold}
Let $Z$ be the ring of integers in a totally real number field $k$. Let $\theta$ be an effective indexing of $Z[X]$; then every $\theta$-computably enumerable relation over $Z[X]$ is a {\it finite-fold} Diophantine relation over $Z[X]$.
\end{theorem}
Thus we concentrate on proving the following proposition.
\begin{proposition}
\label{prop:indexing}
 The set 
\[
\{(n, P_n(X))| n \in \Z_{>0}\}
\]
 is finite-fold Diophantine over $O_k[X]$.
\end{proposition}
The lemmas below constitute a proof of the proposition.    Like the earlier authors, we will make use of a theorem of Y. Pourchet representing positive-definite polynomials as sums of five squares. We start with an auxiliary lemma.
\begin{lemma}
\label{le:fix}
Fix a positive integer $n$ and $n+1$ algebraic integers $\{a_0, a_1, \ldots , a_n\} \subset O_k$.  Then there is exacly one polynomial $G(t) \in O_k[t]$ of degree $n$ such that $G(i)=a_i, i=0,\ldots, n$.   
\end{lemma}
\begin{proof}
Let $G(t)=b_0+b_1t +b_nT^n$ and observe that our requirement on the values of $G(t)$ implies that the coefficients of $G$ must be the solutions of a linear system $A\bar b=\bar a$, where $A=(c_{i,j}), c_{i,j}=i^j, i,j=0,\ldots, n, (\bar b)^t=(b_0,\ldots, b_n), (\bar a)^t=(a_0, \ldots, a_{n})$.  (Here ``$t$'' denotes transposition.)  Note that $\det(A)$ is a van der Monde determinant, and therfore not equal to 0.  Thus, the system has a unique solution.
\end{proof}
\begin{corollary}
\label{cor:h}
Let $h$ be a fixed positive integer, let $F(t) \in O_k[t]$, let $\Omega$ be the set of all embeddings $\sigma$ of $k$ into its algebraic closure,  let 
\[
B_h=\{(G(t) \in O_k[t]) \land  (\deg{G}\leq \deg{F}) \land (\forall \sigma \in \Omega, \forall  i=0,\ldots, \deg F-1:|\sigma(G(i))|\leq h)\}. 
\]
 Then $B_h$ is finite.
\end{corollary}
\begin{proof}
Let $V$ be the set of elements $v$ of $O_k$ such that for any $\sigma \in \Omega$ we have that  $|\sigma(v)| < h$.  Let $m=[k:\Q]$, and let $v \in V$.  Then any coefficient of the monic irreducible polynomial of $v$ over $\Q$ must be an integer of absolute value less or equal to $\max(mh, mh^m)$.    Thus $V$ is a finite set.  

Now let $G(t) \in B_h$.  Then $\deg{G} \leq \deg{F}=n$.  Next we note that  $G(i) \in V$ for $i=0,\ldots,n-1$.  Thus, $(G(0),\ldots, G(n-1)) \in V^n$.  So that the set of possible $n$-tuples of values $\{(G(0),\ldots, G(n-1)) \}$ is finite. By  Lemma \ref{le:fix}, for each $n$-tuple $(a_0,\ldots,a_n)$, there exists  only one polynomial of degree less or equal to $n$ such that $G(i)=a_i$.  Thus we now conclude that the number of polynomials $G$ in $B_h$ is finite. 
\end{proof}
\begin{definition}
\begin{itemize}
\item If $F$ is a polynomial in $O_k[t]$, then  $F$ is {\it positive-definite} on
$k$ (denoted by $\mbox{Pos}(F)$) if and only if $\sigma(F(t)) \geq 0$ for all $t \in k$ and for all real embeddings
$\sigma$ of $k$ into its algebraic closure. 
\item If $F$ is a polynomial in $O_k[t]$, then  $F$ is {\it strictly positive-definite} on
$k$  if and only if $\sigma(F(t)) > 0$ for all $t \in k$ and for all real embeddings
$\sigma$ of $k$ into its algebraic closure. 
\item Let Pos$_2(g,F) \subset \Z \times O_k[t]$, contain pairs $(g,F)$ such that $g^2F=F_1^2+\ldots +F_5$ for some $F_1, \ldots , F_5 \in O_k[t]$.
\end{itemize}
\end{definition}
\begin{lemma}
For any $F \in O_k[t]$ we have that Pos$(F)$ if and only if there exists $g \in \Z$ such that Pos$_2(g,F)$.
\end{lemma}
\begin{proof}
Suppose there exists $g \in \Z$ such that $g^2F=F_1^2+\ldots +F_5^2$ for some $F_1, \ldots , F_5 \in O_k[t]$.  Then, clearly Pos$(F)$ is true.  Conversely, suppose Pos$(F)$ is true.  Then by a theorem of Pourchet (see \cite{Pourchet71}), we have that $F=G_1^2+\ldots +G_5^2$, for some $G_1, \ldots , G_5 \in k[t]$.  Let $g\in \Z$ be such that for any coefficient $a$ of $G_1, \ldots, G_5$, we have that $ga \in O_k$.  Then Pos$_2(g, F)$.
\end{proof}
\begin{lemma}
 The relation Pos$_2$ is finite-fold Diophantine over $O_k[t]$.
\end{lemma}
\begin{proof}
 By definition of Pos$_2$ we have that $\mbox{Pos}_2(g,F)$ if and only if there exist $F_1,\ldots, F_5 \in O_k[t]$ such that
\begin{equation}
\label{fpos}
g^2F=F_1^2+\ldots F_5^2.
\end{equation}
We now show that the for a given $g$ and $F$, there can be only finitely many solutions to \eqref{fpos}.  First of all, the degrees of $F_1,\ldots,F_5$ are bounded by the degree of $F$.  Secondly, observe that for $i=0,\ldots, \deg F-1, j=1,\ldots, 5$ we have that $|\sigma(F_j(i))|^2 \leq g^2|\sigma(F(i))|$ for  all embeddings $\sigma \in \Omega$. (See Corollary \ref{cor:h} for definition of $\Omega$.)  Let $h=\max\{g\sqrt{|\sigma(F(i))|}\}$, where $\max$ is taken over $i=0,\ldots, \deg{F}-1$ and all $\sigma \in \Omega$.  Then if \eqref{fpos} holds, we have that $F_1,\ldots, F_5 \in B_h$. By Corollary \ref{cor:h}, we have that $B_h$ is finite.
\end{proof}
\begin{definition}
The relation $\mbox{Par}(n, b, c, d, v_1, \dots, v_r)$ on the rational integers is defined
to be the conjunction of the following conditions:
\begin{enumerate}
\item $n \in \Z_{\geq 1}$, ($\theta(n)=P_n \in O_k[T]$);
\item $b, c, d, g \in \Z_{\geq 0}, v_1,\ldots,v_r \in \Z$;
\item $d = \deg{P_n}$;
\item \label{par:4} $c$ is the smallest possible non-negative integer such that that $W^2_{d+2} + c -P^2_n - 1$ is strictly positive;
\item  $g$ is the smallest possible positive integer such that Pos$_2(g, W^2_{d+2} + c -P^2_n - 1)$. 
\item \label{par:5} $\forall x \in \Z: \mbox{ if } 0 \leq x\leq d \mbox{ then } W_{d+2}(x) \leq b$;
\item $P_n(2b + 2c + d) = v_1\alpha_1 + \ldots + v_r\alpha_r$.
\end{enumerate}
\end{definition}
\begin{lemma}
\label{le:Par}
Par is a recursive relation on integers.
\end{lemma}
\begin{proof}
Given our assumption that $\theta(n)$ is effective, that is we can effectively determine the degree and coefficients of $P_n$, the first three conditions can be checked algorithmically over $\Z$.  
So, we may start with describing an algorithm  for computing $c$ and $g$. 

  We start by calculating $c$.  For every $\sigma: k \rightarrow \R$ we will determine the smallest non-negative integer $c_{\sigma}$ such that  $\sigma(W^2_{d+2}-P^2_n -1)+c_{\sigma}>0$ for all values of the variable.  Then we will set $c=\max{_{\sigma}}\{c_{\sigma}\}$.
  
  We  compute $c_{\mbox{id}}$ first.  The degree of $P_n$ is $d$, and the degree of $W_{d+2}$ is $d+1$ by Lemma \ref{Pell1}. (We can compute $W_{d+2}$ algorithmically, since $(U_{d+2}-\sqrt{X^2-1}W_{d+2})=(X-\sqrt{X^2-1})^{d+2}$).  Thus the polynomial $W^2_{d+2}-P^2_n -1$ is of degree $2(d+1)$ with a  leading coefficient equal to the square of an element of $k \subset \R$, and therefore has an absolute minimum.  By Corollary \ref{cor:negative min}, there is an algorithm to verify whether this minimum is positive.  If the answer is ``yes'', then we set $c_{\mbox{id}}=0$.  If the answer is ``no'', then we consider $W^2_{d+2}-P^2_n -1+1= W^2_{d+2}-P^2_n$ and check whether $W^2_{d+2}-P^2_n$ is strictly positive for all values of the variable.  If the answer is ``yes'', we set $c_{\mbox{id}}=1$.  If the answer is ``no'', we consider $W^2_{d+2}-P^2_n+1$, etc.  If $\mu<0$ is the minimum value of $W^2_{d+2}-P^2_n -1$, then the process will terminate in at most $[\mu]+1$ steps.  
  
  We now calculate  $c_{\sigma}$ for some $\sigma \ne \mbox{id}$.  We note that since the leading coefficient of $W^2_{d+2}-P^2_n -1$ is a square, the leading coefficient of $\sigma(W^2_{d+2}-P^2_n -1)$ is positive.  Thus, we can proceed as in the case of $\sigma=\mbox{id}$. 
 
We can now determine $g$.   By a result of Pourchet cited above, we can write the polynomial $Y^2_{d+2}-P^2_n+c-1=G_1^2 +\ldots+G_5^2$, where $G_i \in k[T]$.  By examining coefficients of these polynomials,  we can determine an integer $g_{\max}$ such that $g_{\max}G_i \in O_{k}[T]$.  The value $g_{\max}$ is an upper bound on the set of $g$'s we have to search to find $g_{\min}$.   Let 
\[
g^2_{\min}(W^2_{d+2}-P^2_n+ c-1) =F_1^2+\ldots + F_5^2, 
\]
where $F_i \in O_k[T]$.  Then $\deg{F_s} \leq \frac{1}{2}\deg {(W^2_{d+2}-P^2_n+ c-1)}$, and  for $i=0,\ldots,d+1$ and all embeddings  $\sigma$ of $k$ into $\R$, we have that 
\[
|\sigma(F_s(i))| \leq g^2_{\min}(W_{d+2}^2(i)-\sigma(P^2_n(i))+c-1) \leq g_{\max}^2(W_{d+2}^2(i)-\sigma(P_n^2(i))+c-1), 
\]
where $s=1,\ldots , 5$.  Hence, by Corollary \ref{cor:h}, there are only finitely many polynomials $F$ satisfying these inequalities, and we can determine them all.   Once we determine all the possible $F_1,\ldots, F_5$, starting with $g=1$ and continuing through $g_{\max}$, we can check if any quintuple of possible polynomials works with any particular $g$, and thus determine $g_{\min}$.

We now consider Condition \eqref{par:5}.  By checking all the values  of $W_{d+2}(x)$ for $x \in \Z, x \in [1,\ldots,d]$, we can determine the maximum value of the set.  Finally, to determine $v_1,\ldots, v_r$, we can start running through all linear combinations with integer coefficents of the basis vectors until we hit $P_n(2b+2c+d)$. 

\end{proof}
The final piece of proof comes from the lemma below, which is taken essentially verbatim from Zahidi's paper.
\begin{lemma}
\label{le:equal}
$F \in O_k[T] \land  F = P_n$ is equivalent to $\exists n, b, c, d, g, v_1, \ldots, v_r \in O_k[T]$ :
\begin{enumerate}
\item \label{it:par} $\mbox{Par}(n, b, c, d, g, v_1, \ldots,  v_r)$;
\item \label{it:pos} Pos$_2(g,(W^2_{d+2} + c -F^2 - 1))$; 
\item \label{it:val}  $F(2b + 2c + d) = v_1\alpha_1 + \ldots v_r\alpha_r$.
\end{enumerate}
\end{lemma}
\begin{proof}
Suppose $F = P_n$ for some natural number $n$. Then one can easily find
natural numbers $b, c,d, g$ and rational integers $v_1,\ldots, v_r$ such that the relation \eqref{it:par}
is satisfied.  \eqref{it:pos} can be satisfied because $\deg{P_n} < \deg {W_{d+2}}$.

Conversely suppose Conditions \eqref{it:par}-\eqref{it:val} are satisfied for some natural numbers $c, d, n, b$ and integers $v_1,\ldots,  v_r$.  In this case we have to prove that $F = P_n$. From
conditions \eqref{it:par} and \eqref{it:val} it follows that 
\[
(F- P_n)(2b + 2c+d) = 0.
\]
  Thus, if $F \not = P_n$, there is some $S \in O_k[T]\not =0$ such that 
\[
F -P_n = (T-2b+2c+d)S(T).
\]
Now by Condition \eqref{it:pos}, it is the case that $F$ has degree at most $d + 1$, while $P_n$ has degree $d$ (by
Condition \eqref{it:par}), and hence, $S$ has degree at most $d$. So for some integer $k$ with $0\leq k \leq d$, we
have $S(k) \not = 0$. Now for at least one real embedding $\sigma$ we have
\[
|\sigma((F - P_n)(k))| = |(2b + 2c+d-k)||\sigma(S(k))| \geq 2b + 2c,
\]
(since $k \leq d$ and the fact that given an algebraic integer $a$ in a totally real number
field, $a \not= 0$, there is at least one real embedding such that $|\sigma(a)| \geq 1$).
At the same time, again by Condition \eqref{it:pos} of the lemma and by Part  \eqref{par:5} of the definition of the relation Par,  for any real embedding $\sigma$ we have, for all integers $x$ with  $0 \leq x \leq d$ :
\[
|\sigma(F(x))| \leq |\sigma(F(x)^2 + 1)| \leq W^2_{d+2}(x) + c < b+c,
\]
and
\[
|\sigma(P_n(x))|\leq |\sigma(P^2_n(x) + 1) | \leq W^2_{d+2}(x) + c < b+c,
\]
 and hence
\[
|\sigma((F - P_n)(x))|< 2b+ 2c,
\]
leading to a contradiction.
\end{proof}
The last lemma completes the proof of Proposition \ref{prop:indexing} and Theorem \ref{Zahidi-finitefold}.\\

One reason for our emphasis on finite-fold Diophantine definitions is that they
allow us to determine the difficulty of deciding whether a polynomial has
infinitely many solutions in a given ring $S$, an infinite version
of Hilbert's Tenth Problem:
$$HTP^\infty(S)=\bigcup_n\{ f\in S[X_1,\ldots,X_n]~:~(\exists^\infty (x_1,\ldots,x_n)\in S^n)~f(x_1,\ldots,x_n)=0\}.$$
The following corollary is a good example of the connection between these topics:
Theorem \ref{Zahidi-finitefold} actually proves that for the rings involved,
$HTP^{\infty}$ has the greatest complexity possible.

\begin{corollary}
\label{cor-finitefold}
Let $R$ be the ring of integers in a totally real number field $k$. Assume $R[T]$ has an effective indexing $\theta$
with domain $\N$ for which $\theta^{-1}(0)$ is decidable.
Then, for some fixed $n\in\N$, the set of polynomials in $n$ variables over $R[T]$ with infinitely many solutions there,
$$ HTP_n^\infty(R[T])=\{ f\in R[T][X_1,\ldots,X_n]~:~(\exists^\infty (x_1,\ldots,x_n)\in (R[T])^n)~f(x_1,\ldots,x_n)=0\},$$
is $\Pi^0_2$-complete, and thus computably isomorphic to $\overline{\emptyset''}$,
the complement of the Turing jump of the Halting Problem.
\end{corollary}
\begin{proof}
The assumption that $\theta^{-1}(0)$ is decidable allows us to build an effective injective index,
so assume that $\theta$ itself is injective.
By Theorem \ref{Zahidi-finitefold}, there is a polynomial $g(E,Y,Z_1,\ldots,Z_m)$ over $R[T]$
such that, for all $(e,y)\in\N^2$, if $\phi_e(y)$ fails to halt, then $g(e,\theta(y),\vec{Z})$ has no solution in $R[T]$;
whereas if $\phi_e(y)$ does halt, then $g(e,\theta(y),\vec{Z})$ has at least one solution in $R[T]$, but only finitely many.
It follows that, for each fixed $e$, $g(e,Y,\vec{Z})$ has infinitely many solutions in $R$ if $\dom{\phi_e}$ is infinite,
but only finitely many if $\dom{\phi_e}$ is finite.  It is well known that the set $\textbf{Inf}$ of indices $e$ for which
$\phi_e$ has infinite domain is a $\Pi_2$-complete set (\cite{Soare87}, Theorem IV.3.2), computably isomorphic to $\overline{\emptyset''}$,
and we have just described a $1$-reduction from $\textbf{Inf}$ to $HTP_n^\infty(R[T])$, by $e\mapsto g(e,Y,\vec{Z})$.
On the other hand, $HTP_n^{\infty}(R[T])$ itself is $\Pi_2$, hence $1$-reducible to $\textbf{Inf}$,
so by Myhill's Theorem (see \cite[I.5.4]{Soare87}), the two are computably isomorphic.
\end{proof}
The computability-theoretic notation used here is standard; e.g.\ see \cite{Soare87}.
Computable isomorphism is the strongest equivalence in general use in computability,
so the corollary gives a very precise measurement of the complexity of $HTP_n^{\infty}(R[T])$.
The value of $n$ is simply one greater than the least number $m$ of variables required
for a polynomial $g$ giving a finite-fold Diophantine definition of the Halting Problem in $R[T]$.
Notice that we not only have proven the $\Pi_2$-completeness of $HTP^{\infty}(R[T])$,
the general question of whether a polynomial has infinitely many solutions in $R[T]$,
but in fact have established $\Pi_2$-completeness for its restriction $HTP_n^{\infty}(R[T])$
to polynomials with at most $n$ variables.

\section{Defining Valuation Rings over Function Fields of Characteristic 0 }
\label{section:order}
\setcounter{equation}{0}
In this section we give an existential definition of valuation rings for function fields of characteristic 0 over some classes of fields of constants including all number fields.  
More specifically, we will assume the constant field  $k$ to be a field algebraic over $\Q$ with an embedding into a finite extension $M$ of $\Q_p$ for some odd rational prime $p$ or into $\R$ (making $k$ formally real). Note that number fields satisfy these assumptions on $k$.

The method we use below is extendible to a much larger class of fields of characteristic 0 and to higher transcendence degree fields of positive characteristic.  We intend to describe these extensions in future papers.  \\

We now state the two main theorems of the section.
\begin{theorem}
\label{thm: p-order}
Let $k$ be a field algebraic over $\Q$, and such that $k$ has an embedding into a field $M$, a finite extension of $\Q_p$ for some odd prime $p$.  Let $K$ be a function field over $k$, and let $\uu$ be a prime of $K$.  Then the set $\{f \in K: \ord_{\uu} f \geq 0\}$ is Diophantine over $K$.
\end{theorem}

\begin{theorem}
\label{thm: r-order}
Let $k$ be a field algebraic over $\Q$, and such that $k$ has an embedding into $\R$.  Let $K$ be a function field over $k$, and let $\uu$ be a prime of $K$ such that its residue field is embeddable into $\R$.  (For example, $\uu$ can be a prime of odd degree.)  Then the set $\{f \in K: \ord_{\uu} f \geq 0\}$ is Diophantine over $K$.
\end{theorem}

In this section we 

We will need a sequence of lemmas and propositions below before completing the proof.
The first lemma describes a general property of Diophantine definitions.
\begin{lemma}
\label{le:d}
Let $F_2/F_1$ be a finite field extension.  Let $A \subset F_2$ and assume $A$ is Diophantine over $F_2$.  Then $A\cap F_1$ is Diophantine over $F_1$.
\end{lemma}
\begin{proof}
The proof follows from Lemma 2.1.17 of \cite{Sh34} and the fact that $F_2 \leq_{Dioph} F_1$.  (See Definition 2.1.5 of \cite{Sh34}).
\end{proof}
\begin{lemma}
\label{le:notmatter}
Let $\hat K$ be a finite extension of $K$, let $\hat \qq$ be a prime of $\hat K$ above a prime $\qq$ of $K$, and let $A$ be the set of all elements of $\hat K$ with non-negative order at $\hat \qq$.  Assume further that $A$ is Diophantine over $\hat K$.  Then $A \cap K$ is Diophantine over $K$ and consists of all elements of $K$ with non-negative order at $\qq$. 
\end{lemma}
\begin{proof}
Let $f \in K$.  Then $\ord_{\hat \qq}f=\frac{1}{e}\ord_{\qq}f$, where $e=e(\hat \qq/\qq)$.  Therefore, $\ord_{\hat \qq}f \geq 0$ if and only if $\ord_{\qq}f \geq 0$. 
\end{proof}
\begin{remark}
\label{le:notmatter2}
If $k$ has an embedding into a finite extension of $\Q_p$, the same is true of any finite extension of $k$.  
\end{remark}

In view of Lemma \ref{le:notmatter} and Remark \ref{le:notmatter2}, we can assume that if $k$ has an embedding into a finite extension of $\Q_p, p>2$, $k$ has no real embeddings.  (For example, we can adjoin $i$ to the original field.)
\begin{proposition}
\label{prop:aniso}
Let $k$ be any field of characteristic 0 such that the following form:
\begin{equation}
\label{eq:aniso}
X^2-aY^2 -bZ^2 +abW^2
\end{equation}
 is anisotropic over $k$ for some values $a, b \in k$.  If $K$ is a function field over $k$, $\Tt$ is a prime (or a valuation) of $K$ of degree 1 and  $h \in K$ is such that $\ord_{\Tt}h$ is odd, then 
 \begin{equation}
 \label{norm:1}
X^2-aY^2 -bZ^2 +abW^2=h
\end{equation}
has no solution  in $K$.
\end{proposition}
\begin{proof}
Assume the opposite and observe that  due to the fact that function field valuations are non-archimedean and $\ord_{\Tt}h$ is odd,
\[
2\min(\ord_{\Tt}X ,\ord_{\Tt}Y ,\ord_{\Tt}Z , \ord_{\Tt}W) < \ord_{\Tt}h.
\]
 Next let $U\in \{X, Y, Z, W\}$ be such that $\ord_{\Tt}U = \min\{\ord_{\Tt}X, \ord_{\Tt}Y, \ord_{\Tt}X,\ord_{\Tt}W\}$ and divide every variable in \eqref{norm:1} by $U^2$
 \begin{equation}
 \label{norm:2}
\left (\frac{X}{U} \right )^2-a\left (\frac{Y}{U}\right )^2 -b\left (\frac{Z}{U}\right )^2 +ab\left (\frac{W}{U}\right )^2=\frac{h}{U^2}.
\end{equation}
 Observe that $\displaystyle \frac{h}{U^2}$ has a zero at $\Tt$, while at least one of $\displaystyle \{\frac{X}{U}, \frac{Y}{U}, \frac{Z}{U}, \frac{W}{U}\}$ is equal to 1.   Thus considering \eqref{norm:2}  mod $\Tt$, taking into account that $\Tt$ is a degree one prime, we  conclude that the form \eqref{eq:aniso} is isotropic over $k$ in contradiction of our assumption.
 \end{proof}

\begin{notation}
 Since $k$ is embeddable into a finite extension $M$ of $\Q_p$,  we will sometimes identify $k$ with a  subfield of $M$.     Let $\pp$ lie above $p$ in $M$. We identify $\pp$ with the ideal of $M$ containing all elements of positive order at $\pp$.  Let $O_{\pp}$ be the set of elements of non-negative order at $\pp$.  Let $R_{\pp}$ be the residue field of $\pp$.  Since $M$ is complete under the valuation generated by $\pp$, we must have $\Q_p \subseteq k_{\pp} \subseteq  M$.  Thus, we can assume, without loss of generality, that $M=k_{\pp}$.  Let $H \subset k$ be a number field.  Then $O_{\pp}\cap H=O_{\pp_{H}}$ is a valuation ring of some prime $\pp_{H}$ of $H$.  
 
\end{notation}

\begin{lemma}
\label{le:exist}
There exist a number field $H \subseteq k$ containing algebraic integers $a, b$ such that $a$ is not a square in $M,$ $\ord_{\pp}a =0$, and $\ord_{\pp}b$ is odd.
\end{lemma}
\begin{proof}
First of all,  we  claim that there exists a number field $G \subset k$ such that the residue field $R_{\pp_G}\cong R_{\pp}$.  Indeed, since $[M:\Q_p]< \infty,$ we have that $R_{\pp}$ is a finite field, and for any $G$ we have that $[R_{\pp}:R_{\pp_{G}}]$ is a finite extension.  Let $\alpha$ generate $R_{\pp}$ over $R_{\pp_{G}}$.  Since $\alpha \in R_{\pp}$, there exists $x \in M$ such that the residue class of $x$ is $\alpha$.  If $x \not \in k$, since $M=k_{\pp}$, we have that there exists $y \in k$ with $x-y \in \pp$.  Hence the residue class of $y$ is also $\alpha$.  Therefore, in $H=G(y)$, the residue field of $\pp_{H}$ is isomorphic to $R_{\pp}$.  \\

 Now let $\gamma \in M$ and be such that $\ord_{\pp}(\gamma)=0$, and $\gamma$ is not a square modulo $\pp$. Such an element $\gamma$ exists in $M$, because not all residue classes of a finite field $R_{\pp}$ are squares of other residue classes.   Then the number field $H$ contains an element $a$ such that $a \equiv \gamma \bmod \pp$.  Further, any residue class of the prime $\pp_{H}$ contains algebraic integers.  Therefore, we can assume $a \in O_H$.  The existence of $b \in O_H$ such that $\ord_{\pp}b$ is odd can be proved in a similar fashion.

\end{proof}

\begin{lemma}
\label{le:norm}
If $a$ is not a square in $k$, then the form \eqref{eq:aniso} is isotropic over $k$ if and only if  there exists $y \in k(\sqrt{a})$ such that ${\mathbf N}_{k(\sqrt{a})/k}(y)=b$.
\end{lemma}
\begin{proof}
Suppose we have a non-trivial representation of 0 by the form 
\begin{equation}
 \label{norm:0}
X^2-aY^2 -bZ^2 +abW^2.
\end{equation}

Then without loss of generality, we can assume that $Z$ and $W$ are not simultaneously 0.  Otherwise, we are looking at the equation  
\begin{equation}
\label{eq:ns}
X^2 -aY^2=0,
\end{equation}
 while $a$ is not a square in $k$.  The only solution to \eqref{eq:ns} is  $X=Y=0$.  So we get a trivial representation of 0.  

Assuming that $Z$ and $W$ are not simultaneously 0, we note that $Z^2-aW^2 \ne 0$,  and we can rewrite \eqref{eq:aniso} as 
\[
\frac{X^2-aY^2}{Z^2-aW^2}=b
\]
or
\begin{equation}
\label{knorm}
\exists y \in k(\sqrt{a}): b=N_{k(\sqrt{a})/k}(y).
\end{equation}

Conversely, suppose \eqref{knorm} is true.  Then $b=U^2-aV^2$, where either $U \ne 0$, or $V \ne 0$.  Thus we have that $U^2 - aV^2 -b=0$.  Let $X=U, Y=V, Z=1, W=0$, and we have a non-trivial representation of 0 by the form \eqref{norm:0} over $k$.

\end{proof}

\begin{lemma}
 Let $H \subset k$ be a number field.  Let $a, b \in O_H$ be such that $\ord_{\pp}a=0$, $a$ is not a square in $M$ and $\ord_{\pp}b$ is odd.  (Such a number field $H$ and elements $a, b \in H$ exist by Lemma \ref{le:exist}.)   Then  \eqref{eq:aniso} is anisotropic over $k$ and $M$.  
\end{lemma}
\begin{proof}
Suppose \eqref{norm:0} is isotropic.  Then by Lemma \ref{le:norm}, we have that \eqref{knorm} is true.  Since $a$ is not a square  in $M$, $\ord_{\pp}a=0$, $\pp$ is not a dyadic prime, we have that the extension $M(\sqrt{a})/M$ is an unramified extension of degree 2 of  $\pp$-adically complete fields.   Since $M$ is a complete field, there is only one prime above $\pp$ in $M(\sqrt{a})$.  If $\ttt$ is the prime of $M(\sqrt{a})$ above $\pp$, then  the relative degree $f(\ttt/\pp)$ of $\ttt$ over $\pp$ satisfies $f(\ttt/\pp)=2$.   Thus,
\[
\ord_{\pp}N_{k(\sqrt{a})/k}(y)=f(\ttt/\pp)\ord_{\ttt}y \equiv 0 \bmod 2.
\]   
But $\ord_{\pp}b$ is odd.  So \eqref{knorm} cannot hold, and the form \eqref{norm:0} is anisotropic.  
\end{proof}
\begin{lemma}
\label{le:units}
Let $H$ be a number field.  Let $\qq$ be any prime of $H$.  Let $a, b \in H$  be units at $\qq$. Assume further that $a  \equiv 1 \bmod 4$.  Then  in $H_{\qq}$ the form \eqref{norm:0} is isotropic.  
\end{lemma}
\begin{proof}
If  $\qq$ is not dyadic, and $a$ is a unit at $\qq$, then $\qq$ is unramified in the extensions $H(\sqrt{a})/H$ and $H_{\qq}(\sqrt{a})/H_{\qq}$.  If $\qq$ is a dyadic prime, then given our assumption that $a \equiv 1 \bmod 4$, we also have that $\qq$ is unramified in the extensions $H(\sqrt{a})/H$ and $H_{\qq}(\sqrt{a})/H_{\qq}$. Further, since $b$ is a unit at $\qq$ also, it is a norm in the extension of $H_{\qq}(\sqrt{a})/H_{\qq}$, by local class field theory.  Hence  \eqref{knorm} can be solved in $H_{\qq}$.
\end{proof}
\begin{corollary}
\label{even power}
Let $ \qq, a, b, H$ be as above, but  assume $a$ is a unit at $\qq$ while $\ord_{\qq}b$ is even.  Then  in $H_{\qq}$ the form \eqref{norm:1} is isotropic.  
\end{corollary}
\begin{proof}
Let $\pi$ be a local uniformizing parameter with respect to $\qq$ in $H$.  If $\ord_{\qq}b$ is even, then we can replace  $b$ in the quadratic form by $\hat b=\pi^{2s}b$, where $s \in \Z$ and $\ord_{\qq}\pi^{2s}=-\ord_{\pp} b,$ without changing the status of the form with respect to being isotropic or anisotropic.  Observe that $\hat b$ is a unit at $\qq$, and the Lemma follows from Lemma \ref{le:units}.
\end{proof}

\begin{lemma}[Essentially Eisenstein Irreducibility Criteria]
\label{lemma:eisen}
Let $H, \qq$ be as above. Let  $a_0,\ldots, a_n \in H$,  be such that $\ord_{\qq}a_m=0$, $\ord_{\qq}a_i \geq r >1$, for $i=1, \ldots, m-1, \ord_{\qq}a_0 = r-1$ with $(m,r-1)=1$.   Let
\[
f(T)=a_mT^m +a_{m-1}T^{m-1} +\ldots + a_0 \in H[T]
\]
 In this case $f(T)$ is irreducible over $H_{\qq}$ and adjoining a root of $f(T)$ produces a totally ramified extension of $H_{\qq}$.
\end{lemma}
\begin{proof}
Let $\alpha$ be a root of $f(T)$ in the algebraic closure of $H_{\qq}$.  Let $\Qq$ be a prime above $\qq$ in $H_{\qq}(\alpha)$.  
If $\ord_{\Qq} \alpha \geq 0$, then, since $\ord_{\qq}a_m=\ord_{\Qq}a_m=0$, we have that
\[
\ord_{\Qq}\alpha^m=\ord_{\Qq}(-a_{m-1}\alpha^{m-1} - \ldots -a_1\alpha -a_0).
\]
Let $e=e(\Qq/\qq)$, and
note that
\[
\ord_{\Qq}(a_i\alpha^i)>e \cdot \ord_{\qq}a_i \geq er>e(r-1)=e \cdot \ord_{\qq}a_0.
\]
Thus, 
\[
\ord_{\Qq}\alpha^m=\ord_{\Qq}(-a_{m-1}\alpha^{m-1} - \ldots -a_1\alpha -a_0)=e \cdot \ord_{\qq}a_0=e(r-1)
\]
Since $(r-1,m)=1$, we conclude that $e=m$, and the polynomial is irreducible.

 Suppose now $\ord_{\Qq}\alpha <0$.  Then for any $i=1,\ldots, m-1$ we have 
 \[
 \ord_{\Qq}(\alpha^m) < \ord_{\Qq}(\alpha^i)<\ord_{\Qq}a_i\alpha^i.
 \]
 Therefore,  
 \[
 e(r-1)=\ord_{\Qq}a_0 =\ord_{\Qq}(-\alpha^m -a_{m-1}\alpha^{m-1}-\ldots -a_1\alpha)=m \cdot \ord_{\Qq}(\alpha),
 \]
 so that once again we have that $e=m$,  and $f(T)$ is irreducible.
 
\end{proof}
We now establish existence of an element of $K$ with a particular pole divisor.
\begin{lemma}
\label{le:T}
For any valuation $\uu$ of $K$, for all sufficiently large $s>0$, there exists $T \in K$ such that $\uu$ is the only valuation where $T$ has a pole, and $\ord_{\uu}T=-3^s$.
\end{lemma}
\begin{proof}
One can use the same proof as in Lemma 3.5 of \cite{Sh92}, except that one should replace $2$ by $3$.
\end{proof}
For the convenience of the reader we state a proposition due to Y. Pourchet (see \cite{Pourchet71}, Proposition 3) concerning representation of  polynomials by quadratic forms over rational function fields.
\begin{proposition}
\label{prop:Pourchet}
Let  $k$ be a field of characteristic $\ne 2$.   Let $a, b  \in k, f \in k[T]$, where $T$ is transcendental over $k$.  Then there exist $X, Y, W, Z \in k[T]$ such that $f=X^2-aY^2-bZ^2+abW^2$ if and only if  the following conditions are satisfied 
\begin{itemize}
\item For every prime factor $p(T)$ of  $f(T)$ over $k$ of odd multiplicity, we have that the form is isotropic in the residue field of  $k(T)$ modulo $p(t)$.
\item The form represents the leading coefficient of $f(T)$ over $k$.
\end{itemize}
\end{proposition}

\begin{proposition}
\label{prop:p-adic}
Let $k$ be a field  such that $k$ has an embedding into a finite extension $M$ of $\Q_p$ for an odd rational prime $p$.   Let $L=\tilde \Q \cap M$, let $O_L$ be the ring of algebraic integers of $L$, and let $a, b \in O_L$ be relatively prime.  Let $\pp$ lie above a prime $p$ in $M$,  assume  that  $a$ is not a square at $\pp$ and $\ord_{\pp}b$ is odd.  Assume further that $a\equiv 1 \bmod 4$.   (These assumptions can be realized since $k$ can be embedded into a finite extension of $\Q_p$ for an odd $p$, and by the Strong Approximation Theorem.)  Let $\uu$ be a prime of $K$ of degree 1,  let $T$ be transcendental over $k$, and let $K$ be a finite extension of $k(T)$, so that $k$ is relatively algebraically closed in $K$.  Assume further that $\ord_{\uu}T=-3^s$, and $T$ has no other poles. Let $f \in K$.  
\begin{enumerate}
\item If $\ord_{\uu}f <0$ and $\ord_{\uu}f$ is odd, then for any $\xi \not=0, \mu \in k$, the equation 
\begin{equation}
 \label{norm:3}
X^2-aY^2 -bZ^2 +abW^2=\xi f +\mu
\end{equation}
has no solution in $K$.
\item Let $H \subset k$ be a number field containing $a, b$.  If $f \in H[T]$, and $\ord_{\uu}f$ is even (or in other words $\deg f$ is even), then there exist $\xi\not=0, \mu \in H$ such that \eqref{norm:3} has a solution in $H(T) \subset K$.
\item For each $f$, there exists a finite set $\calQ$ of primes of $H$ and a positive constant $N$ such that if $\xi_1, \mu_1 \in H$ and are such that $\ord_{\qq}(\xi-\xi_1)>N, \ord_{\qq}(\mu-\mu_1)>N$ for all $\qq$, then 
\begin{equation}
 \label{norm:3.1}
X^2-aY^2 -bZ^2 +abW^2=\xi_1 f +\mu_1
\end{equation}
has a solution $(X,Y, Z, W) \in  K^4$.

\end{enumerate}
\end{proposition}
\begin{proof}
Suppose $f \in K$ and $\ord_{\uu}f $ is odd.    In this case, for any $\xi\not=0,  \mu \in k$ we know that $\ord_{\uu} (\xi f  +\mu)=\ord_{\uu}f <0$ and  it is odd. So, by Proposition \ref{prop:aniso} applied to $h=\xi f +\mu$ we conclude that \eqref{norm:3} has no solutions in $K$.

If we consider the form $X^2-aY^2 -bZ^2 +abW^2$ over $H$, then we note that any four dimensional form is universal locally at any non-archimedean prime $\qq$, i.e. it represents every element of the field. Without loss of generality, by Lemma \ref{le:d} we can assume that $k$, and therefore $H$ have no real embeddings. By the Hasse-Minkowski local-global principle, we can conclude that the form is universal over $H$.  Thus, if $f \in H$, (and therefore $\ord_{\uu}f=0$), the equation \eqref{norm:3} can be satisfied.

Now assume that $f \in H[T]\setminus H$, and $\ord_{\uu}f$ is even (or in other words $\deg f$ is even).   We now show that for some constants $\xi \not=0, \mu \in H$ the quadratic form equation \eqref{norm:3} has solutions in $H(T)$.  

We start with examining the isotropic/anisotropic status of \eqref{norm:0} over $H$. If the form is isotropic in $H$, then by Proposition \ref{prop:Pourchet},  we are done, since the form can represent any constant in $H$.  Suppose the form is anisotropic over $H$. Then, by the Hasse-Minkowski Theorem,  for some prime  $\qq$ of $H$, the form is anisotropic over $H_{\qq}$.  Further, since for all but finitely many primes $\qq$ of $H$ we have that $a, b$ are units at $\qq$, we have that the form \eqref{norm:0} is isotropic over $H_{\qq}$ for all but finately many $\qq$, by Lemma \ref{le:units}.

 We   now describe the set of conditions on $\xi$ and $\mu$ making sure that 
 \begin{enumerate}
 \item $\xi f +\mu$ is irreducible over $H(T)$, and
 \item   if $\ttt$ is a prime of $H(T)$ corresponding to $\xi f+\mu$, then in the residue field of $\ttt$, the norm \eqref{norm:0} becomes isotropic.
 \end{enumerate}
  Let $\qq$ be a prime of $H$ such that $\ord_{\qq}b$ is odd,  and $a$ is not a square$\mod \qq$ (in other words, $\qq$ is one of finitely many primes where the quadratic form \eqref{norm:0} is anisotropic in $H_{\qq}$). Let $\pi_{\qq} \in O_H$ be an element of order 1 at $\qq$. (Such an element exists by the Weak Approximation Theorem.)  Let $a_0,\ldots, a_{n} \in H$ and assume
\[
f(T)=a_nT^n+\ldots +a_0=a_n\left (\frac{\pi_{\qq}^r T}{\pi_{\qq}^r}\right)^n +a_{n-1}\left(\frac{\pi_{\qq}^r T}{\pi_{\qq}^r}\right)^{n-1} +\ldots + a_0=
\]
\[
a_n\left (\frac{U}{\pi_{\qq}^r}\right)^n +a_{n-1}\left (\frac{U}{\pi_{\qq}^r}\right)^{n-1} +\ldots + a_0,
\]
where $r$ is a non-negative integer such that $\displaystyle \ord_{\qq}\pi_{\qq}^r\frac{a_i}{a_n}>2$   for any coefficient $a_i, i=0,\ldots, n-1$, and $U=\pi_{\qq}^rT$.
Now set $\displaystyle \xi_{\qq}=\frac{\pi_{\qq}^{nr}}{a_n}$ and let $g(U)=\xi_{\qq}f(T)$.  (Observe that $\xi_{\qq} \ne 0$ as required.)  Let
\[
g(U)=U^n+c_{n-1}U^{n-1} +\ldots +c_0.
\]
Then
\[
\displaystyle c_i=\xi_{\qq }\frac{a_i}{\pi_{\qq}^{ri}}=\frac{(\pi_{\qq})^{nr}}{a_n}\frac{a_i}{\pi_{\qq}^{ri}}=\pi_{\qq}^{nr-ri}\frac{a_i}{a_n}=\pi_{\qq}^{nr-ri-r}\frac{\pi_{\qq}^r a_i}{a_n}
\]
Thus, 
\[
\ord_{\qq}c_i >2+(nr-ri-r)\geq 2.
\] 
for $i=1,\ldots,n-1$, and $\ord_{\qq}c_0 > (n-1)r+2$.  

Next let $\ord_{\qq}\mu_{\qq}=1$ and observe that $h(U)=g(U)+ \mu_q$ is irreducible by Lemma \ref{lemma:eisen}, and adjoining a root $\alpha$ of $h(U)$ to $H$ will ramify $\qq$ with even ramification degree.  This will make the quadratic form in \eqref{norm:1} isotropic at the factors above $\qq$ in $H(\alpha)$ by Corollary \ref{even power}.  

Let $\calQ$ be the set of all primes $\qq$ of $H$ such that the form \eqref{norm:0} is anisotropic over $H_{\qq}$.  Let $\pi \in O_H$ be such that $\ord_{\qq}\pi=1$ for all $\qq \in \calQ$.  Now let $\xi=\frac{\pi^r}{a_n}$, where $r$ is a non-negative integer such that for all $\qq \in \calQ$ we have that $\displaystyle \ord_{\qq}\pi^r\frac{a_i}{a_n}>2$   for any coefficient $a_i, i=0,\ldots, n-1$.  Let $\mu \in H$ be such that $\ord_{\qq}\mu=1$ for all $\qq \in \calQ$.  Such elements $\pi, \mu$ exist in $H$ by the Weak Approximation Theorem, once again.  Now, by the argument above, we see that in $H(\alpha)$ the form \eqref{norm:0} is isotropic locally at all primes of $H(\alpha)$, and therefore the form is isotropic over $H(\alpha)$, by the Hasse-Minkowski Theorem.  

As discussed above, we can assume that $H$ has no real embeddings and therefore a four-dimensional quadratic form  is universal over $H$.  So the leading coefficient of $h(U)$ is represented by the form. 
  We now apply Proposition \ref{prop:Pourchet} to reach the desired conclusion.\\
  
Now let  $\xi_1, \mu_1 \in H$, let $\hat g(U)=\xi_1f(T)=\hat c_nU^n+\hat c_{n-1}U^{n-1} +\ldots +\hat c_0$.   Then $\hat c_i=\xi_1\frac{a_i}{\pi^{ri}}$.  Let $\qq \in \calQ$ and  note that $\ord_{\qq}(c_i-\hat c_i)=\ord_{\qq}(\xi-\xi_1)+ \ord_{\qq}\frac{a_i}{\pi^{ri}}$.  Let $M>0$ be large enough so that 
  \[
  M +\min_{i,\qq}(\ord_{\qq}\frac{a_i}{\pi^{ri}}) >\ord_{\qq}(\xi\frac{a_i}{\pi^{ri}})=\ord_{\qq}c_i.
  \]
    Pick $\xi_1$ such that $\ord_{\qq}(\xi-\xi_1) >M$ for all $\qq \in \calQ$ and observe that for all $\qq \in \calQ$, we now have that $\ord_{\qq}(\hat c_i -c_i) > \ord_{\qq} c_i$.   Then for all $\qq \in \calQ$, we get that 
    \[
    \ord_{\qq}\hat c_i=\ord_{\qq}((\hat c_i-c_i) +c_i)=\ord_{\qq} c_i \geq 2
    \]
 for $i >1$,    and $\ord_{\qq}\hat c_n=0$.  Now let $\mu_1$ be such that for all $\qq$ we have that 
 \[
 \ord_{\qq}(\mu-\mu_1)>2.  
 \]
 Then $\ord_{\qq}\mu_1=\ord_{\qq}\mu$ by the same argument as for $\xi_1$.  Thus, $\ord_{\qq}\mu_1=1$.  We now can apply Lemma \ref{lemma:eisen} to the polynomial $\hat h=\xi_1f(T)+\mu_1$ to conclude that all primes in $\calQ$ ramify in the residue field of extension generated by $\hat h$ ramification degree divisible by $2$ and the proceed in the same manner as we did for polynomial $h$.
  \end{proof}

We now consider the case of $k$ with a real embedding.  First we prove the following lemma.
\begin{lemma}
\label{le:unr}
Let $H$ be a number field such that $\sqrt{2} \in H$.  Let $\ttt$ be a dyadic prime of $H$.  Then $-1 \in H_{\ttt}$, and therefore the extension $H(i)/H$ is unramified at all primes of $H$.
\end{lemma}
\begin{proof}
Note that $(1+\sqrt{2})^2=3+2\sqrt{2}\equiv -1 \bmod 2\sqrt{2}$.  So if we set $g(x)=x^2+1$, then $g(1+\sqrt{2})\equiv 0 \bmod 2\sqrt{2}$, while $g'(1+\sqrt{2})=2+2\sqrt{2} \equiv 0 \bmod 2$.  Thus, by Hensel's Lemma, $g(x)$ has a root in $H_{\ttt}$.  Hence, in the extension $H(i)/H$ the local degree at $\ttt$ is 1, and therefore $\ttt$ is not ramified in this extension.  

Since dyadic primes are the only primes possibly ramified in the extension $H(i)/H$, we see that no prime of $H$ is ramified in the extension $H(i)/H$.

\end{proof}
Without loss of generality, by Lemma \ref{le:d}, we can assume that $k$ contains the square root of $2$. (Adjoining the square root of $2$ will not change the existence of a real embedding.)  Let $L=M \cap \tilde \Q$, as above.  Let $H \subset L$ be a number field containing a square root of 2.  We now prove an analog of Proposition \ref{prop:p-adic} for fields with a real embeddings.
 
\begin{proposition}
\label{prop:r}
Let $k$ be a field with a real embedding and containing a square root of 2.  Let $T$ be transcendental over $k$ and let $K$ be a finite extension of $k(T)$ so that $k$ is relatively algebraically closed in $K$.  Let $\uu$ be a prime of $K$ of degree 1,  and assume  $\ord_{\uu}T=-3^s$, and $T$ has no other poles.  Let $f \in K$.  
\begin{enumerate}
\item If $\ord_{\uu}f <0$ and is odd, then for any $\xi\not=0, \mu \in k$, the equation 
\begin{equation}
 \label{norm:4}
X^2+Y^2 +Z^2 +W^2=\xi f +\mu
\end{equation}
has no solution in $K$.
\item If for some number field $H \subseteq k$ containing $\sqrt{2}$, we have that $f \in H[T]$ and $\ord_{\uu}f$ is even (or in other words $\deg f$ is even), then there exist $\xi\not=0, \mu  \in k$ such that \eqref{norm:4} has a solution.
\item If $\xi, \mu \in k$ are as above, then there exists $\delta >0$ such that for all $\xi_1, \mu_1$ with $|\xi-\xi_1| <\delta, |\mu-\mu_1| < \delta$ for all archimidean absolute values $|...|$ of $k$, then 
\begin{equation}
\label{norm:4.1}
X^2+Y^2 +Z^2 +W^2=\xi_1 f +\mu_1
\end{equation}
has solutions in $K$.
\end{enumerate}
\end{proposition}
\begin{proof}
We first note that the quadratic form in \ref{norm:4} is anisotropic over $\R$, and therefore anisotropic over $k$.  Thus, if $\ord_{\uu}f$ is odd, then the proposition holds by the same argument  as in Proposition \ref{prop:p-adic}.  The same is true if $f \in H$.  (By the Weak Approximation Theorem, we can always pick $\xi, \mu\in H$ so that the image of $f$ under any real embedding is positive.)

So, assume now that $\deg f$ is even and $f$ is not a constant.  It is enough to show that \eqref{norm:4} can be satisfied in $H(T)$.  We again start by examining anisotropic/isotropic status of the form in question over $H$.  As in Lemma \ref{le:d}, it is easy to see that the quadratic form in \eqref{norm:4} is isotropic if and only if $-1$ is a norm in the extension $H(i)/H$.   By the Hasse-Minkowski Theorem, it is enough to have that $-1$ is a norm in all completions of $H$.  Since, $H$ contains $\sqrt{2}$, by Lemma \ref{le:unr}, we have that the extension $H_{\qq}(i)/H_{\qq}$ is unramified for all primes $\qq$ of $H$, and therefore $-1$ is a norm by local class field theory at all primes $\qq$ of $H$.  Thus, the only completion, where $-1$ is not a norm is the real one.

 We choose $\xi$ so that the leading coefficient of $\xi f$ is positive under all real embeddings.  Such a $\xi$ exists by the Weak Approximation Theorem, once again.  This step will also make sure that the leading coefficient of $\xi f+\mu$ is representable by the form over $H$.  We also choose $\mu >0$ large enough so that $\xi f +\mu$ has no roots in $\R$ under all real embeddings.  Let $h(T)=\xi f(T)+\mu$, and let $g(T)$ be an irreducible factor of $h(T)$.  Then $g(T)$ has no roots in $\R$ under any real embedding of $H$, and therefore must be of even degree.  Further if we adjoin a root $\alpha$ of $g(T)$ to $H$, the extended field  $H(\alpha)$ will have no real embeddings, and the left side of \eqref{norm:4} will become isotropic.  Thus we can apply Proposition \ref{prop:Pourchet} again to reach the desired conclusion.
 
 If $\xi_1$ is sufficiently close to $\xi$ under all archimedean absolute values of $k$, then the leading coefficient of $\xi_1 f$ is also positive under all real embeddings.  Similarly, if $\mu_1$ is sufficiently close to $\mu$ under all archimedean valuations of $k$, then $h_1(T)=\xi_1 f(T)+\mu_1$ has no real roots under all real embeddings of $k$.
\end{proof}

We now address the issue of giving a Diophantine definition of a set of constants guaranteed to contain constants $\xi_1, \mu_1$ we used in  Propositions \ref{prop:p-adic} and \ref{prop:r}.
\begin{proposition}
\label{prop:constants}
The following statements are true.
\begin{enumerate}
\item \label{it:1} There exists a Diophantine over $K$ set of constants $A$ such that for any number field $H \subseteq k$ and any finite collection $\qq_1, \ldots, \qq_r$ of primes of $H$, the set $\{(b,\ldots, b), b\in A \cap H\} \subset A^r$ is dense in $H_{\qq_1} \times \ldots  \times H_{\qq_r}$ under the product topology.
\item \label{it:2} There exists a Diophantine over $K$ set of constants $A$ such that for any number field $H \subseteq k$ and all real embeddings $\sigma_1, \ldots, \sigma_r$  of $H$, the set $\{(b,\ldots, b), b \in A \cap H\}\subset A^r$ is dense in $\sigma_1(H)\times \ldots  \times \sigma_r(H)$ under the product topology.

\end{enumerate}
\end{proposition}
\begin{proof}
   For the $p$-adic  case the proof follows from Theorem 5.5 of \cite{Eispadic}.   
For the Archimedean case, we use Lemma 3.6 and Section 3.6 of \cite{PS} together with Proposition 5.1 of \cite{Eispadic}.   
  \end{proof}
  
Before completing the proof of Theorems \ref{thm: p-order} and \ref{thm: r-order}, we need the following lemmas.
\begin{lemma}
\label{le:const}
Let $\qq$ be a prime divisor of $K$ of degree greater than 1.  Then there exists a finite constant field extension $\hat{k}$  of $k$ such that in $\hat{k}K$ the divisor $\qq$ has at least one factor of degree 1.
\end{lemma}
\begin{proof}
Let $R_{\qq}$ be the residue field of $\qq$.   Then $R_{\qq}$ is isomorphic to a finite extension of $k$, and we can identify $R_{\qq}$ with this extension.   Let $s(T) \in k[T]$ be the monic irreducible polynomial of a generator $\alpha$ of $R_{\qq}$ over $k$.   Let $\qq_i$ be a factor of $\qq$ in $K(\alpha)$.  Let $R_i$ be the residue field of $\qq_i$. Since the power basis of $\alpha$ is an integral basis with respect to $\qq$, we can determine the factorization of $\qq$ in $K(\alpha)$ by considering the factorization of $s(T)$ over $R_{\qq}$ (see \cite{L}, Chapter I, Section 8, Proposition 25).  By assumption on $s(t)$ we have that over $R_{\qq}$ it has at least one factor of degree 1.
\end{proof}
\begin{lemma}
\label{le:has a pole}
Let $h, g \in K$ be such that for some prime $\aaa$ of  $K$, it is the case that  $\ord_{\aaa}g=-3^s$, and $\ord_{\aaa}(h^{3^{s+1}}+g)$  is  either non-negative or even,  Then, $\ord_{\aaa}h <0$, and $\ord_{\aaa}h \equiv 0 \bmod 2$.
\end{lemma}
\begin{proof}
First assume that $\ord_{\aaa}h \geq 0$.  Then $\ord_{\aaa}(h^{3^{s+1}}+g)=\ord_{\aaa}g=-3^s$, contradicting our assumptions.  Thus $\ord_{\aaa}h<0$, and $\ord_{\aaa}h^{3^{s+1}} < \ord_{\aaa}g$.  Consequently, $\ord_{\aaa}(h^{3^{s+1}}+g)=\ord_{\aaa}h^{3^{s+1}} \equiv 0 \bmod 2$, and the conclusion of the lemma holds.
\end{proof}

We now complete the proof of  Theorem \ref{thm: p-order}.  
\subsection{Proof of Theorem \ref{thm: p-order}}
\subsubsection{A special element $T$}
\label{sub:T}
Let $k$ be algebraic over $\Q$ and embeddable into a finite extension $M$ of $\Q_p$.  Without loss of generality we can assume that $k$ does not embed into $\R$, by adjoining, if necessary, a square root of $-1$ to $k$. A finite extension of $k$ will continue to be embeddable into a finite extension of $\Q_p$.  Further, by Lemma \ref{le:d}, we can translate Diophantine definitions obtained for the extended field into Diophantine definitions over the original field.  

Next let $\uu$ be a prime of $K$.  By Lemmas \ref{le:d} and \ref{le:const}, using the same reasoning as above, we can assume  that $\uu$  is of degree 1.  Let $T \in K\setminus k$.  For all sufficiently large $s \in \Z_{>0}$, by Lemma \ref{le:T}, we can find a $T \in K$ that has a pole at $\uu$ of order $3^s$, and no other poles. Note that $\uu$ is totally ramified over $k(T)$ with the ramification degree $3^s$.   Hence, for any $h \in k(T)$ we have that $3^s \cdot \deg {h}=-\ord_{\uu}h$, and $\deg h \equiv 0 \bmod 2$ if and only if $\ord_{\uu}h \equiv 0 \bmod 2$.

\subsubsection{Defining a subset of $K$ containing all polynomials in $T$ of even degree, and no element of $K$ with an odd degree pole at $\uu$}
\label{sub:pol}
Let $h \in k[T]$.  Then by Proposition \ref{prop:p-adic} and Proposition \ref{prop:constants}, Part \ref{it:1}, we have that $\ord_{\uu}h \equiv 0 \bmod 2$ or equivalently $\deg h \equiv 0 \bmod 2$ if and only if  there exists $\xi, \mu, X, Y, Z, W \in K$ such that 
\begin{align}
 \label{eq:h}
X^2-aY^2 -bZ^2 +abW^2=\xi h +\mu,\\
\label{eq:h2}
\xi \in A \setminus \{0\}, \mu \in A.
\end{align}
We remind the reader that $A \subset k$ is Diophantine over $K$.
At the same time, if $h \in K$ satisfies \eqref{eq:h}, then either $h$ has no pole at $\uu$ or this pole is of even order.

\subsubsection{Defining a subset of $K$ containing all rational functions in $T$ of even degree and no element of $K$ of odd order at $\uu$}
Please note that for any $h \in k(T) $, we have that $\ord_{\uu}h \equiv 0 \bmod 2$  if and only if $\displaystyle h=\frac{h_1}{h_2}$, where $ h_1, h_2 \in k[T] \setminus k$, $h_2(T) \ne 0$ and $\deg {h_1} \equiv \deg {h_2} \equiv 0 \bmod 2$.  Indeed, suppose $\displaystyle h=\frac{g_1}{g_2}$, $g_2 \ne 0, g_1, g_2 \in k[T]\setminus k$.  If $\deg h$ is even, then $\deg {g_1} -\deg {g_2}$ is even.  Suppose $\deg {g_1}=2r+1,  r\in \Z_{>0}$.  Then note that $\displaystyle h=\frac{Tg_1}{Tg_2}$, where we can set $h_1=Tg_1, h_2=Tg_2$ to reach the desired conclusion.  If either $g_1$ or $g_2$ is a constant, then let $h_1=T^2g_1, h_2=T^2g_2$ to reach the desired conclusion once again.

Therefore, $h \in k(T), \deg h\equiv 0 \bmod 2$ if and only if 
there exists $\xi_{r}, \mu_{r}, X_{r}, Y_{r},  Z_{r}, W_{r} \in K$ such that 
\begin{align}
 \label{eq: h1}
X_{r}^2-aY_{r}^2 -bZ_{r}^2 +abW_{r}^2=\xi_{r} (h^{3^{s+1}}_r +T) +\mu_{r},\\
\xi_{r} \in A \setminus \{0\}, \mu_{r} \in A.
\end{align}
where $r=1,2$, $h_2 \ne 0$ and $\displaystyle h =\frac{h_1}{h_2}$.
Indeed, by Proposition \ref{prop:aniso}, we have that $\deg {h_1^{3^{s+1}} +T}, r=1,2$ is either even or non-positive.  By Lemma \ref{le:has a pole}, we then conclude that $\deg {h_r }>0$ and is even.  Consequently, $\deg h$ is even.
Conversely, if $\deg h$ is even, we can write $h$ as a ratio of two polynomials $h_1, h_2$ of positive even degrees.  If $\deg {h_r}$ is positive and even, then $\deg {h_r^{3^{s+1}}+T}$ is also positive and even, and therefore we can satisfy \eqref{eq: h1} over $K$ by Proposition \ref{prop:p-adic}.
Finally, Proposition \ref{prop:p-adic} and Lemma  \ref{le:has a pole} imply that  for any $h_1, h_2 \in K$, we have that \eqref{eq: h1} implies that $h_1, h_2$ have an even order pole at $\uu$.

\subsubsection{Defining a subset of $K$ containing all rational functions of $T$ integral at $\uu$ (or of non-positive degree) and no functions of $K$ with a pole at $\uu$}
Let $f \in k(T)$ be such that $\deg {f^{2\cdot 3^{s+1}}T+T^2} \equiv 0 \bmod 2$.  We claim that in this case $\deg f \leq 0$.  Suppose not.  Then $\deg{f^{2\cdot 3^{s+1}}T} > \deg{T^2}$, and 
\[
\deg{f^{2\cdot 3^{s+1}}T+T^2}=\deg{f^{2\cdot 3^{s+1}}T}=2\deg{f^{3^{s+1}}}+1 \equiv 1 \bmod 2.
\]
At the same time, if $\deg f \leq 0$, then $\deg{f^{2\cdot 3^{s+1}}T} < \deg{T^2}$, and  
\[
\deg{f^{2\cdot 3^{s+1}}T+T^2}=\deg{T^2}=2\equiv 0 \bmod 2.
\]
Consider now the set of $f \in K$ satisfying the following equations.

\begin{align}
 f^{2\cdot 3^{s+1}}T+T^2=\frac{h_1}{h_2}   \label{eq:begin1}\\
X_{j,r}^2-aY_{j,r}^2 -bZ_{j,r}^2 +abW_{j,r}^2=\xi_{j,r} (h^{3^{s+1}}_r+T) +\mu_{j,r}, j,r=1,2, \\
 \xi_{i,r} \in A \setminus \{0\}, \mu_{i,r} \in A. \label{eq:end1} 
 \end{align}

By Proposition \ref{prop:aniso} and Lemma \ref{le:has a pole}, we have that  $\ord_{\uu}h_r\equiv 0 \bmod 2$.  Therefore, 
\[
\ord_{\uu}(f^{2\cdot 3^{s+1}}T+T^2) \equiv 0 \bmod 2,
\]
 and thus $\ord_{\uu}f\geq0$.  At the same time, if $f \in k(T)$, and $\deg f$ is even, then we can choose $h_r \in k[T]$ to be of even positive degree, and by Proposition \ref{prop:p-adic}, we can satisfy \eqref{eq:begin1}--\eqref{eq:end1}.

\subsubsection{Defining the valuation ring $V_{\uu}$ of $\uu$ in $K$}
Let $n := [K:k(T)]=3^s$.  We claim that $V_{\uu}$ can be defined as follows: $w \in V_{\uu}$ if and only if  there exist 
\[
\xi_{i,r}, \mu_{i,r}, a_i, h_{i,r} , X_{i,r}, Y_{i,r}, Z_{i,r}, W_{i,r} \in K
\]
 with   $r \in \{1,2\} , i \in \{0,\ldots, n-1\}$ such that
\begin{equation}
\label{eq:begin}
w^n +a_{n-1}w^{n-1} +\ldots + a_0=0~\&
\end{equation}
\begin{equation}
\bigwedge_{i=0}^{n-1}(( Ta_i^{2\cdot 3^{s+1}}+T^2)=\frac{h_{i,1}}{h_{i,2}})~\&
\end{equation}
\begin{equation}
\label{eq:middle}
\bigwedge_{i=0}^{n-1} \bigwedge_{r=1}^2 X_{i,r}^2-aY_{i,r}^2 -bZ_{i,r}^2 +abW_{i,r}^2=\xi_{i,r} (h_{i,r}^{3^{s+1}} +T)+\mu_{i,r } ~\&
\end{equation}
\begin{equation}
\label{eq:end}
\forall i \in \{0,\ldots, n-1\}, r \in \{1,2\} ~ (\xi_{i,r} \in A \setminus \{0\} ~\& ~\mu_{i,r} \in A).
\end{equation}
First suppose that Equations \eqref{eq:begin}--\eqref{eq:end} are satisfied.  Then, since $A \subset k$, and $\xi_{i,r} \ne 0$, by Proposition \ref{norm:1}, for any degree 1 prime $\aaa$ of $K$, we have that  $\ord_{\aaa} (\xi_{i,r} (h_{i,r}^{3^{s+1}} +T))$ is even.  Therefore, $\ord_{\aaa} (h_{i,r}^{3^{s+1}}+T)$ is either bigger or equal to 0, or is even.   
In particular, we have that $\ord_{\uu}(h_{i,r}^{3^{s+1}}+T) \geq 0$ or $\ord_{\uu}(h_{i,r}^{3^{s+1}} +T) \equiv 0 \bmod 2$.  Now by Lemma \ref{le:has a pole}, we conclude that $\ord_{\uu}h_{i,r}\equiv 0 \bmod 2$ (and $\ord_{\uu}h_{i,r}<0$).  Thus,   $\ord_{\uu}(Ta_i^{2\cdot 3^{s+1}}+T^2)$ is even, and we now have  that $\forall i\in \{0,\ldots.n-1\} ~\ord_{\uu}a_i \geq 0$.  

Suppose now that $\ord_{\uu}w<0$.  In this case, 
\[
\ord_{\uu}(w^n +a_{n-1}w^{n-1} +\ldots + a_0)=n\ord_{\uu}w <0,
\]
 contradicting the fact that 
 \[
 \ord_{\uu}(w^n +a_{n-1}w^{n-1} +\ldots + a_0)=\ord_{\uu}0=\infty. 
 \]
 Therefore, if for some $w \in K$ we have that Equations \eqref{eq:begin}--\eqref{eq:end} can be satisfied over $K$, then $\ord_{\uu}w \geq 0$.  

Conversely, suppose $w \in K, \ord_{\uu}w\geq 0$.  Then $w$ is integral  over the local  subring $R_{1/T}$ of $k(T)$ containing all functions $a \in k(T)$ without a pole at the valuation that is the pole of $T$ in $k(T)$.  In other words, $R_{1/T}$ consists of all rational functions $a \in K(T)$ with $\deg a \leq 0$, or, equivalently $\ord_{\uu}a\geq 0$.  Hence, $w$ will be a root of a polynomial in \eqref{eq:begin}, such that $a_i \in R_{1/T}$.  If $\ord_{\uu}a_i \geq 0$, then $\ord_{\uu}(T\cdot a_i^{2\cdot 3^{s+1}}+T^2)$ is a rational function in $T$ of even degree.  Consequently, we can write each $\displaystyle a_i=\frac{h_{i,1}}{h_{i,2}}$, where $h_{i,r}$ are polynomials in $T$ of even degrees.  Thus, $h_{i,r}^{3^{s+1}} +T$ is a polynomial in $T$ of even degree, and we can find constants in $A$ so that \eqref{eq:middle} can be satisfied over $K$.  This concludes the proof of Theorem \ref{thm: p-order}.\\

We now proceed to prove Theorem \ref{thm: r-order}.
\subsection{Proof of Theorem  \ref{thm: r-order}}
In almost every way the proof of Theorem \ref{thm: r-order} is the same as the proof of Theorem \ref{thm: p-order}.  So we confine ourselves to discussing only those parts where there are differences. We will consider every part of the proof of Theorem \ref{thm: p-order} and indicate what changes, if any, are required.

	We start with examining Subsection \ref{sub:T}.  In this part of the proof we consider a prime $\uu$ and determine whether we can assume that $\uu$ is of degree 1.  For Theorem \ref{thm: r-order}, we consider only primes $\uu$ with residue fields embeddable into $\R$.  If $\uu$ is not of degree 1, then its residue field is isomorphic to a  finite extension $\hat k$ of $k$. By Lemma \ref{le:const}, in the extension $\hat kK$ of $K$ the prime $\uu$ will have a factor of degree 1.  As in Subsection \ref{sub:T}, we can extend our field of constants $k$, but the extended field must still be embeddable into $\R$.  This condition will be satisfied for $\hat k$, given our assumptions on $\uu$.  Thus, as in the proof of Theorem \ref{thm: p-order}, we can assume that $\uu$ is of degree 1. We can again produce an element $T \in K$ such that $\uu$ is the only pole of $T$ and $\ord_{\uu}T=-3^s$.
	
	From this point on, the proof of Theorem \ref{thm: r-order} is exactly the same as the proof of Theorem \ref{thm: p-order} with \eqref{eq:h} replacing \eqref{norm:4}, and the set $A$ defined to satisfy the conditions of Proposition \ref{prop:r}.
  The existence of such a set $A$  follows from Proposition \ref{prop:constants}\eqref{it:2}.  	

\section{Diophantine Definition of C.E. Sets over Rings of Integral Functions}
\label{sec:gence}
In this section we extend results of J.\ Demeyer to show that c.e.\ sets are definable over any ring of integral functions, assuming the constant field is a number field.  Since Demeyer showed that such a result holds over polynomial rings over number fields, and since rings of integral functions are finitely generated modules over polynomial rings, it is enough to show that we can give a Diophantine definition of polynomial rings over the rings of integral functions to achieve the desired result.  Below we state the main theorem of the section.
\begin{theorem}
\label{thm:extce}
Let $K$ be a function field over a field of constants that is a finite extension of $\Q$, and let $\calS$ be a finite non-empty collection of its valuations. Then every c.e.\ subset of $O_{K,\calS}$ is Diophantine over $O_{K,\calS}$.
\end{theorem}

\subsection{Arbitrary powers of a ring element}
In this section we again turn our attention to the rings of $\calS$-integers of function  fields, discussed in Sections \ref{sec:function} and \ref{intsingle}, and consider the case where the field is not necessarily rational.  We recall the notation and assumptions we used in these sections and add new ones.
\begin{notationassumptions}\textup{}
\begin{itemize}  
\item  Let $K, k$, $\calS$, $a, \qq, \qq_{\infty}, T=a-\sqrt{a^2-1},$ be as in Proposition \ref{prop:exist}.
\item Let $R=O_{K,\calS}, R'=R[T], R''=R[T] \cap O_{K(T), \{\qq_{\infty}\}}$.  Since $T$ satisfies a monic polynomial of degree 2 over $R$, every element of $R[T]$ is of the form $ a+bT$, where $a, b \in R$.
\item Let $\calS'$ be the set of primes of $K(T)$ lying above primes of $\calS$.
\item Let $\Qq_{\infty}$ be the prime below $\qq_{\infty}$ in $\Q(T)$.  In other words $\Qq_{\infty}$ corresponds to the infinite valuation of $\Q(T)$.  Observe that $\qq_{\infty}$ is the only prime above $\Qq_{\infty}$ in $K(T)$.
\item Let $d=[K(T):\Q(T)]$.
\item Let $L$ be the Galois closure of $K(T)$ over $\Q(T)$.
\item Let $m=[L:\Q(T)]$.
\item Let $\beta \in R''$ generate $K(T)$ over $\Q(T)$. Since $\beta \in R''$, we have that $\beta=a_{\beta}+b_{\beta}T$.
\item Let $\ttt_1,\ldots, \ttt_s$ be all the factors of $\qq_{\infty}$ and $\Qq_{\infty}$ in $L$.
\item For each positive integer $m$ let $\xi_m$ be a primitive $m$-th root of unity.
\item For each positive integer $m$ let $\Phi_{m}$ be the monic irreducible polynomial of $\xi_{m}$ over $\Q$.  We refer to polynomials of this form as ``cyclotomic'' polynomials.
\end{itemize}
\end{notationassumptions}
\subsection{Outline of the proof}
 For the results below we  need to construct a Diophantine definition of arbitrary powers of a non-constant element of the ring.  It turns out that it is more convenient to construct this definition for an element $T$ of a quadratic extension $R'$ of the ring $R$.  We will proceed as follows.
\begin{enumerate}
\item  Using Proposition \ref{prop:exist} we show that the set 
\[
P(T)=\{(a, b) \in R^2| \exists m \in \Z_{\geq 0}: a+bT=T^m\}
\]
 is Diophantine over $R$.
\item Next  we use the set $P(T)$ to show that the set
\[
Z(T)=\{(a, b) \in R^2| a+bT \in \Z[T]\}
\]
is Diophantine over $R$.  This is the main technical result of the section.
\item \label{it:start here}At this point, using results of J. Denef,  we deduce that for any c.e. set $A \subset \Z(T)^r$, the set of the form 
\[
\{(a_1,b_1, \ldots, a_r,b_r)|a_i, b_i \in R, (a_1+b_1T, \ldots, a_r +b_rT) \in A\}
\]
 is Diophantine over $R$. 
\item Let $Y \in R, Y\not \in k$.   Now observe that the set of $2d$-tuples 
\[
A_Y=(c_0(T), u_0(T),\ldots, c_{d-1}(T), u_{d-1}(T)) \subset \Z[T]^{2d}
\]
 such that $u_0(T)\ldots u_{d-1}(T) \not=0$ and $\displaystyle \sum_{i=0}^{d-1}\frac{c_i(T)}{u_i(T)}\beta^i \in \Z[Y]$ is computable, and therefore c.e.  Hence, the set
\[
\{(a_0,b_0, \ldots, a_{2d-2},b_{2d-2})|a_i, b_i \in R, \sum_{i=1}^{d}\frac{a_{2i-2}+b_{2i-2}T}{a_{2i-1}+b_{2i-1}T}\beta^{i-1} \in \Z[Y]\}
\]
is Diophantine over  $\Z[T]$, and therefore over $R$.  We can replace $\beta^i$ by $a_{\beta, i}+b_{\beta, i}T$ with $a_{\beta, i}, b_{\beta,i}\in R$, since $\beta^i \in R''$.   Further, using the fact that the conjugate of $T$ over $K$ is $T^{-1}$, and $T+T^{-1} =2a \in R$, we can rewrite
\[
\frac{a_{2i-2}+b_{2i-2}T}{a_{2i-1}+b_{2i-1}T}=\frac{(a_{2i-2}+b_{2i-2}T)(a_{2i-1}+b_{2i-1}T^{-1})}{a_{2i-1}^2+b_{2i-1}^2+2aa_{2i-1}b_{2i-1}}.
\]
\[
=\frac{\hat a_{2i-1}+\hat b_{2i-1}T}{\hat a_{2i}},
\]
where $\hat a_{2i}, \hat b_{2i-1}, \hat a_{2i-1}$ are  polynomials in $a_{2i-2},b_{2i-2}, a_{2i-1}, b_{2i-1} $ with coefficients in $R$ depending on $K$ and $a$ only. 
  
  Then, multiplying out all the products, using the monic irreducible polynomial of $T$ over $K$ to replace all powers of $T$ greater than 1, and since the sum has to be in $K$, at the end we will obtain  that
\[
\sum_{i=1}^{m}\frac{a_{2i-1}+b_{2i-1}T}{a_{2i}+b_{2i}T}\beta^i=\sum_{i=1}^{m}\frac{\hat a_{2i-1}+\hat b_{2i-1}T}{\hat a_{2i}}(a_{\beta,i}+b_{\beta,i}T)=\frac{\tilde a}{\tilde b}, \tilde a, \tilde b \in R,
\]
where $\tilde a=P_a(a_1,\ldots, a_{4d}, b_1, \ldots, b_{4d}), \tilde b=P_b(a_1,\ldots, a_{4d}, b_1, \ldots, b_{4d})$ are fixed polynomials with coefficients in $R$ in $a_{2i-2},b_{2i-2}, a_{2i-1}, b_{2i-1} $ with coefficients depending on $K$ and $a$ only. Finally, since $\frac{\tilde a}{\tilde b}=c \in R$, we conclude that 
$\Z[Y]$ consists of all elements $c \in R$ such that for some $\tilde b \in R$,  we have that 
\[
\tilde bc =P_a(a_1,\ldots, a_{2d-1}, b_1, \ldots, b_{2d-1}), 
\]
\[
\tilde b =P_b(a_0,\ldots, a_{2d-1}, b_0, \ldots, b_{2d-1}),
\]
 and $(a_0,b_0, \ldots, a_{2d-1},b_{2d-1})$ range over a Diophantine subset of $R^{4d}$.  Hence $\Z[Y]$ is Diophantine over $R$.  Now using the result of Denef one more time, we can assert that all c.e. subsets of $R$ are Diophantine. 
\end{enumerate}
To simplify the proof, we will  have $K(T)$-variables range not over $R'=R[T]$ but over a subring $R''$ of $R'$, where only one valuation $\qq_{\infty}$ is allowed as a pole of non-constant elements of the ring.  The following lemma shows that this restriction is a Diophantine condition relative to $R$.
\begin{lemma}
\label{le:doubleprime}
The set $\{ (a,b) \in R^2| a+Tb \in R''\}$ has a Diophantine definition over $R$.
\end{lemma}
\begin{proof}
Observe that once we fix an element $a \in O_{K,\calS}$, the field $K(T)$ is fixed.  The constant field of $K$ is a number field.  Therefore, by Theorem \ref{thm: p-order}, for each $\ttt \in  \calS'$ we have that the valuation ring $V_{\ttt}$ of $\ttt$ has a Diophantine definition over $K(T)$.  Hence, $V_{\ttt} \cap R'$ has a Diophantine definition over $R'$.  Thus, $R''=\bigcap_{\ttt \in \calS' \setminus \{\qq_{\infty}\}}V_{\ttt}$ is existentially definable over $R'$.  Consequently, there exists a polynomial $P(a+bT,z_1, \ldots,z_m)$ with coefficients in $R'$ such that for any $a,b \in R$ the equation $P(a+bT,\bar z)=0$ has solutions $z_1, \ldots,z_m \in R'$ if and only if $a+bT \in R''$.  Using the fact that $1$ and  $T$ are linearly independent over $K$, and $T^2-2a+1=0$, we can replace $P(a+bT,z_1,\ldots,z_m)$ by a polynomial $Q(a,b,v_1,\ldots,v_{r})$ such that for any pair $(a,b )\in R^2$ the equation $Q(a,b,v_1,\ldots,v_{r})=0$ has solutions $v_1,\ldots, v_{r} \in O_{K,\calS}$ if and only if $a+bT \in R''$.
\end{proof}

\subsection{Defining Polynomials over $\Z$ Using Root-of-Unity Polynomials }
In this section review a set polynomials from \cite{Demeyer10} and show how to adapt  these polynomials algebraic extensions of rational function fields. In his paper J.\ Demeyer defined a set $\calC$ of  {\it root-of-unity polynomials} to be the set of polynomials $F\in  \Z[T]$ 
satisfying one of the following three equivalent conditions:
\begin{enumerate}
\item $F$ is a divisor of $T^u -1$ for some $u > 0$.
\item $F$ or $-F$ is a product of distinct cyclotomic polynomials.
\item $F(0) = \pm1$, $F$ is squarefree, and all the zeros of $F$ are roots of unity.
\end{enumerate}
Observe that the constant polynomials $F(T)=\pm1$ satisfy the conditions above.  

The following property of polynomials in $\calC$ will serve as a foundation for applying the "weak vertical method" later on in this section.  The description of the method can be found in \cite{Sh34}.
\begin{proposition}
\label{forweak}
 Let $F \in \Z[T]$ with $F(0) \in \{-1, 1\}$, and let $\ell  \in \Z_{>0}$. In this case there
exists a polynomial $M \in \calD$ such that $F \equiv M \mod T^{\ell}$ in $\Z[T]$.
 \end{proposition}
 \begin{proof}
 Proposition 2.7 of \cite{Demeyer10}.   
 \end{proof}
 We will now prove a technical lemma to be used in determining the value of polynomials for some values of variables.
 \begin{lemma}
 \label{le:equivbelow}
 Let $\beta \in R''$ and assume that $\beta \equiv b \bmod (T-c)$ in $R''$ with $b, c \in \Z$.  Then $N_{K(T)/\Q(T)}(\beta) =G(T) \in \Q[T]$, and $G(c) =\pm b^{d}$.
 \end{lemma}
 \begin{proof}
 First, $R''$ is contained in the integral closure of $\Q[T]$ in $K(T)$.  Therefore, all coefficients of the monic irreducible polynomial of $\beta$ over $\Q(T)$ are polynomials in $\Q[T]$.  Hence, \[G(T)=N_{K(T)/\Q(T)}(\beta)\in \Q[T].\]    Let $R_L$ be the integral closure of $R''$ in $L$ (the Galois closure of $K(T)$ over $\Q(T)$).  Next consider the congruence $\beta \equiv b \bmod (T-c)$ in $R_L$.         Let $\beta_1=\beta,\ldots, \beta_m$ be all the conjugates  of $\beta$ over $\Q(T)$.  Since $T, c, b \in \Q(T)$, we have that $\beta_i \equiv b \bmod (T-c)$ in $R_L$.  Thus, 
 \[
 N_{L/\Q(T)}(\beta) =\prod_{i=1}^m\beta_i \equiv b^m \bmod (T-c)
 \]
  in $R_L\cap \Q(T)=\Q[T]$.  Further,    
 \[
 N_{L/\Q(T)}(\beta)=(N_{K(T)/\Q(T)}(\beta))^{[L:K(T)]}=G(T)^{m/d}.
 \]
   Therefore, $G(T)^{m/d}\equiv b^m \bmod (T-c)$, and $G(c)^{m/d}= b^m$.  Consequently, $G(c)=\xi \cdot b^d$, where  where $\xi$ is a root of unity.  Since $G(c)\in \Q$, we conclude that $G(c)=\pm b^d$.
 \end{proof}
 
In what follows we will need some well-known facts about roots of unity and function fields that the reader can find in the appendix.

\begin{lemma}
\label{alpha}
Suppose $\alpha \in R''$, and let $ m_1,\ldots, m_r, p_1,\ldots, p_r, c \in \Z_{>0}$ be defined as in Lemma \ref{le:approx}.  Let $n_1,\ldots, n_r \in \Z_{>0}$ be such that $\ord_{p_i}\Phi_{m_i}(c)= n_i$.  Let $b \in \Z$ be such that $\ord_{p_i}b=n_i$. Suppose now that  Equations \eqref{a:1}--\eqref{a:6} below hold with variables ranging over $R''$.
\begin{enumerate}
\item \label{a:1} $\alpha | (T^\ell-1)$ in $R''$,
\item \label{a:2} $\alpha \equiv \pm1 \bmod T$.
\item \label{a:3} $\ord_{\qq_{\infty}}\alpha = \ord_{\qq_{\infty}}T^{\sum_{i=1}^r n_i}$.
\item \label{a:6}  $\alpha \equiv b\mod (T-c)$ in $R''$.
\end{enumerate}
In this case $N_{K(T)/\Q(T)}(\alpha)= \prod_{i=1}^{r}\Phi_{m_i}(T)^{d}$, and $\alpha \in \Q(T)$.  Conversely, if $\alpha =\prod_{i=1}^{r}\Phi_{m_i}(T)$, then  \eqref{a:1} -- \eqref{a:6} can be satisfied in the remaining variables over $R''$.
\end{lemma}
\begin{proof}
Let $\QQ_{\infty}$ be the prime below $\qq_{\infty}$ in $\Q[T]$.  By Corollary \ref{cor:one}, the prime $\qq_{\infty}$  is the only prime above $\QQ_{\infty}$ in $K(T)$, and the ramification degree $e$ of $\qq_{\infty}$ over $\QQ_{\infty}$ is equal to $\ord_{\qq_{\infty}}T$.  Therefore, by Proposition \ref{prop:Chev}, we have that 
\[
f(\qq_{\infty}/\QQ_{\infty})=-d/\ord_{\qq_{\infty}}T, 
\]
where $f(\qq_{\infty}/\QQ_{\infty})$ is the relative degree of $\qq_{\infty}$ over $\QQ_{\infty}$.   Next observe that since $\alpha$ has a pole at $\qq_{\infty}$ only, $\alpha$ is integral with respect to $\Q[T]$, and  therefore  $N_{K(T)/\Q(T)}(\alpha)$ is a polynomial over $\Q$ in $T$.  Further,   using the assumption that $\ord_{\qq_{\infty}}\alpha = \ord_{\qq_{\infty}}T^{\sum_{i=1}^rn_i}$ and by Proposition \ref{prop:Chev} again, we have that
\[
\deg {N_{K(T)/\Q(T)}(\alpha)}=- \ord_{\QQ_{\infty}}N_{K(T)/\Q(T)}(\alpha)=-f(\qq_{\infty}/\QQ_{\infty}) \cdot \ord_{\qq_{\infty}}(\alpha)
\]
\[
=\frac{d}{\ord_{\qq_{\infty}}T}\ord_{\qq_{\infty}}(\alpha)=
\frac{d}{\ord_{\qq_{\infty}}T}(\ord_{\qq_{\infty}}T)\sum_{i=1}^rn_i=d \sum_{i=1}^r n_i.
\]

Second, since $\alpha | (T^{\ell}-1)$  in $R''$, we have that  $N_{K(T)/\Q(T)}(\alpha) \big{|} (T^{\ell}-1)^{d}$ in $\Q[T]$.  The polynomial $T^{\ell}-1$ does not have any multiple roots in $\bar \Q$, the algebraic closure of $\Q$.    Thus, the roots of $N_{K/\Q(T)}(\alpha)$ in $\bar \Q$ are of multiplicity at most $d$ and are $\ell$-th roots of unity.    

 Let $G(T):=N_{K(T)/\Q(T)}(\alpha) \in \Q[T]$. 
Then 
\[
G(T) =u\prod_{j| \ell}\Phi_j(T)^{a_j}, 
\]
where $a_j \in \{0,1, \ldots, d\}$ and $u \in \Q$.  By Lemma \ref{le:equivbelow} and Assumption \eqref{a:2}, we have that $G(0)=\pm 1$.
By Lemma \ref{ap1} we also have that $\prod_{j| \ell}\Phi_j(0)^{a_j}=\pm 1$.  Thus, we conclude that $u=\pm1$.   We now note that $\ord_{p_i}G(c)=\sum_{j|\ell}a_j \cdot  \ord_{p_i}\Phi_j(c_i)=a_in_i$ by the assumption on $c$.

From Assumption \eqref{a:6} and Lemma \ref{le:equivbelow}, we have that $G(c) = \pm b^d$.  Consequently,  for all $i=1,\ldots, r$, we have that $\ord_{p_i}G(c)=\ord_{p_i}b^d =d \cdot \ord_{p_i}b \equiv 0 \bmod n_id$.
 Thus, $a_i >0$ and $d\Big{|}a_{n_i}$.  But  $0\leq a_{n_i} \leq d$.  Hence, $a_{n_i} = d$ and 
\[
G(T)=\pm \prod \Phi_{n_i}(T)^{d}. 
\]
 Consequently,  
 \[
 N_{K(T)/\Q(T)}\left (\frac{\alpha}{\prod_{i=1}^r\Phi_{n_i}(T)}\right )=\pm 1,
 \]
 implying that  
 \[
 \frac{\alpha}{\prod_{i=1}^r\Phi_{n_i}(T)}
 \]
  is a unit of $R''$.  But the only units of this ring are elements of the constant field $k$.    Hence $\alpha = \mu \prod_{i=1}^r\Phi_{n_i}(T)$ for some $\mu \in k$.  But by Assumption \ref{a:2}, we have that $\alpha \equiv \pm 1 \bmod T$, and thus $\mu \prod_{i=1}^r\Phi_{n_i}(0) = \pm 1$, implying as before that $\mu =\pm 1$.  
 
    It is clear that $\alpha = \prod_{i=1}^r\Phi_{n_i}(T)$ satisfies \eqref{a:1} -- \eqref{a:6}.
\end{proof}

\bigskip
We now show that all conditions in Lemma \ref{alpha} are Diophantine over $R$, and therefore the set $\calD$ has a Diophantine description over $R$.
\begin{lemma}
\label{le:special}
$\{(a,b) \in R| a+bT \in \calD\}$ is Diophantine over $R$.
\end{lemma}
\begin{proof}
We need to convert our assumptions on $\ell, m, c, b$ and  Conditions \eqref{a:1} -- \eqref{a:6} of Lemma \ref{alpha} into a Diophantine definition of the set $\calD$.  First consider a recursive subset $Z$ of $\Z^3$ satisfying the following condition.
\[
(\ell, c, b, n) \in Z
\]
if and only if
\begin{enumerate}
\item there exist $ r, m_1,\ldots,m_r \in \Z_{>1}$ such that $\ell=m_1\ldots m_r$, $(m_i,m_j)=1$,
\item there exist $p_1\ldots p_r$, where for each $i=1,\ldots,r$ we have that $p_i$ is a prime number, and $p_i-1\equiv 0 \bmod \ell$.
\item $n_i:=\ord_{p_i}\Phi_{m_i}>0$, 
\item For all $i=1,\ldots,r$, for all $j$ such that $\ell \equiv 0 \bmod j$, it is the case that $j \ne m_i$ implies $\ord_{p_i}\Phi_j(c)=0$.
\item $n=\sum_{i=1}^rn_i$,
\item For all $i=1,\ldots,r$, we have that $\ord_{p_i}b=n_i$.
\end{enumerate}
By the MDRP theorem $Z$ is Diophantine over $\Z$ and therefore over $R$.  Further, as we noted above, the set $\{(s, u_s, w_s), s \in \Z_{>0}\},$ where $u_s-\sqrt{a^2-1}w_s =T^s$, is Diophantine over $R$ by Corollary \ref{s,Ts}.  Thus Condition \eqref{a:1} is Diophantine. 
Next we note that $\alpha \in R''$, and  we can  replace Condition \eqref{a:2}  with 
\[
\frac{\alpha -1}{T} \in R'' \lor \frac{\alpha +1}{T} \in R''.  
\]

Further, Condition \eqref{a:3} can be replaced with 
\[
\ord_{\qq_{\infty}}\frac{\alpha}{T^n} \geq 0.
\]
The order  conditions are Diophantine over $R$ as explained in Section \ref{section:order}.   We replace  Condition \eqref{a:6} with the following Diophantine condition: $\frac{\alpha-b}{T-c} \in R''$.
 \end{proof}
\bigskip
We will now use Proposition \ref{forweak} to give an existential definition of all polynomials in $T$ over $\Z$.  We use what elsewhere we called the ``Weak Vertical Method'' (see \cite{Sh34}).  
\begin{lemma}
\label{le:bound}
Let $X \in R'', X \not \in k$, and let $X=\sum_{i=0}^{d-1}c_i \beta^i, c_i \in \Q(T)$.  Then there exist $D=D(\beta) \in \Q[T], C=C(\beta) \in \R_{>0}$ dependent on $\beta$ only, such that $Dc_i =a_i \in \Q[T]$ and for all $i=0,\ldots,d-1$ we have that 
\[
\deg {a_i} < C|\ord_{\qq_{\infty}}X|.
\]
\end{lemma}
\begin{proof}
We proceed via a ``Linear Algebra'' proof of the sort described in Chapter 9 of \cite{Sh34}.  Let $L$, as above, be the Galois closure of $K(T)$ over $\Q(T)$.     Let ${\calT}=\{\ttt_1, \ldots, \ttt_s\}$ be  the set of all distinct factors of $\qq_{\infty}$ in $L$.  Since $\qq_{\infty}$ is the only factor of the infinite prime $\Qq_{\infty}$ of $\Q(T)$ in $K(T)$,  we have that ${\calT}$ also contains all factors of $\Qq_{\infty}$ in $L$.    

{\it Claim:} For any $i=1,\ldots, s$ and any $\sigma \in \Gal L {\Q(T)}$, it is the case that 
\[
\ord_{\ttt_i}\sigma(X)=\ord_{\ttt_1}X.  
\]

{\it Proof of the claim:} Since $\ttt_1,\ldots,\ttt_s$ are conjugates over $K(T)$ and $\Q(T)$, the Galois group
$\Gal L {K(T)}$ acts transitively on the set ${\calT}$,  and all elements of $\Gal L {\Q(T)}$ permute ${\calT}$.
So,  fix $i \in \{1,\ldots,s\}$ and $\sigma \in \Gal L {\Q(T)}$.   For some $r \in \{1,\ldots, s\}$, we have that $\sigma(\ttt_r)=\ttt_i$.  Let $\mu \in \Gal L {K(T)}$ be such that $\mu(\ttt_1)=\ttt_r$.  Then $\ord_{\ttt_1}X=\ord_{\sigma \mu(\ttt_1)}\sigma\mu(X)=\ord_{\ttt_i}\sigma(X)$. Similarly, $\ord_{\ttt_1}\beta=\ord_{\ttt_i}\sigma(X)$ for all $i=1,\ldots,s, \sigma \in \Gal L {\Q(T)}$.

Let $\sigma_1=\mbox{id}, \ldots, \sigma_d \in \Gal L {\Q(T)}$ be such that the set $\{\sigma_1(\beta),\ldots, \sigma_d(\beta)\}$ contains all distinct conjugates of $\beta$ over $\Q(T)$.

Now consider the following system of linear equations.
\[
A\bar a=\bar X,
\]
where 
\[
A=(\sigma_j(\beta^i)), j=1, \ldots, d, i=0,\ldots, d-1, 
\]
\[
\bar a=(c_0,\ldots, c_{d-1})^t, \bar X=(\sigma_1(X),\ldots, \sigma_{d}(X))^t.
\]
(Here ``$t$'' denotes transpose.) Since $\sigma_j(\beta)\ne \sigma_r(\beta)$ for all $j \ne r \in \{1,\ldots, d\}$, we have that $\det(A) \ne 0$ as a Vandermonde determinant.  Using Kramer's Rule, we can solve for $c_0, \ldots, c_{d-1}$ in terms of $\det(A), \sigma_r(X)$ with $r=1,\ldots, d$.  We obtain that 
\[
c_j=\frac{\det A_j}{\det A},
\]
where $A_j$ is the matrix obtained from $A$ by replacing its $j$-th column by the column $(\sigma_1(X), \ldots, \sigma_d(X))^t$.  

  Since  $X, \beta \in R''$, all entries of $A_j$ and $A$ have poles at all factors of $\qq_{\infty}$ in $L$, and no other poles. (This is so because $\qq_{\infty}$ and $\sigma(\qq_{\infty})$ have the same factorization in $L$. )   Therefore, if we set $D=\det^2 (A) \in \Q[T]$, then $a_j=Dc_j=\det A\det A_j \in \Q(T) \cap O_{L, \{\ttt_1,\ldots, \ttt_s\}}=\Q[T]$.    Thus, we have that $\ord_{\ttt_1}a_j <0$.  Let $A_{i,j}$ be the $i,j$-th minor of $A$.  Then $\ord_{\ttt_1}\det A<0, \ord_{\ttt_1}\det A_{i,j}<0$ and these orders depend on $\beta$ only.  Let $C_1=\max_{i,j}(|\ord_{\ttt_1}\det A|, |\ord_{\ttt_1}\det A_{i,j}|)$.  
  
  We now make the following observation we will use in our calculations below.  Let $Y=\sum_{r}Y_r \in L$.  Assume $\ord_{\ttt_1}Y<0$, and for all $r$ we have that $\ord_{\ttt_1}Y_r<0$.  Let $Y^*$ be such that $\ord_{\ttt_1}Y^* =\min_{r}\{\ord_{\ttt_1}Y_r\}$.  Then $\ord_{\ttt_1} Y \geq \ord_{\ttt_1}Y^*$, and  $-\ord_{\ttt_1} Y \leq -\ord_{\ttt_1}Y^*$.  Since $\ord_{\ttt_1} Y <0,  \ord_{\ttt_1}Y^*<0$, it follows that $|\ord_{\ttt_1} Y| \leq |\ord_{\ttt_1}Y^*|$.
  
  Using co-factors along the $j$-th column, we see that $\det A_j =\sum_{i=0}^{d-1}\pm \sigma_{i+1}(X)\det A_{i,j}$.  Further, using the observation above and the fact $\ord_{\ttt_1}\sigma_j(X)=\ord_{\ttt_1}X$, we also conclude that 
  \[
  |\ord_{\ttt_1}\det A_j|=|\ord_{\ttt_1}(\sum_{i=0}^{d-1}\pm \sigma_{i+1}(X)\det A_{i,j})|\leq |\ord_{\ttt_1}X|+C_1<2C_1|\ord_{\ttt_1}X|.
    \]
  Thus
  \[
  |\ord_{\ttt_1}a_j| =|\ord_{\ttt_1}\det A +\ord_{\ttt_1}\det A_j|< C_1 +2C_1|\ord_{\ttt_1}X|<3C_1|\ord_{\ttt_1}X|=C|\ord_{\ttt_1}X|.  
  \]
  Thus,
  \[
  \deg {a_j}=|\ord_{\Qq_{\infty}}a_i| < | \ord_{\qq_{\infty}}a_i| <C|\ord_{\qq_{\infty}}X|,
  \]
where $C=C(\beta)$ depends on $\beta$ only.
  \end{proof}

We now prove  our main technical result of this section.
\begin{proposition}
The set $\{(a,b) \in O_{K,\calS}| a+bT \in \Z[T]\}$ is Diophantine over $R$.
\end{proposition}
\begin{proof}
Let $z=z(\beta) \in \Z_{>0}$ be such that $z(\beta)>C(\beta)$.  We start with the following claim.

{\it Claim:}  Given $Y \in R''\setminus k$, the following system of equations and conditions can be satisfied over $R''$  if only if $Y \in \Z[T]$, and $Y(0)=\pm 1$.
\begin{equation}
\label{eq:calD}
M \in \calD \subset \Q[T], 
\end{equation}
\begin{equation}
\label{eq:order}
\ord_{\qq_{\infty}}\frac{T^{\ell}}{Y^{z}} <0,
\end{equation}
\begin{equation}
\label{eq:equiv}
  Y \equiv M \bmod T^{\ell} \mbox{ in } R''.  
\end{equation}
{\it Proof of the claim:}

 First we assume that the Equations \eqref{eq:calD}--\eqref{eq:equiv} are satisfied.  As in Lemma \ref{le:bound},  we write $Y=\sum_{i=0}^{d-1}c_i\beta^i$, where  $c_i \in \Q(T)$.  Further, by Lemma \ref{le:bound}, we know that $Dc_i \in \Q[T]$, and $\deg {Dc_i} < C(\beta) \ord_{\qq_{\infty}}Y$, Next, we observe that 
\[
Y-M=(c_0-M)+c_1(T) \beta + \ldots +c_{d-1}(T)\beta^{d-1}, 
\]
and 
\[
\frac{Y-M}{T^{\ell}}\in O_{K(T),\{\qq_{\infty}\}},
\]
 by Equation \ref{eq:equiv}.  Further,
\[
\frac{Y-M}{T^{\ell}}=f_0+f_1\beta +\ldots +f_{d-1}\beta^{d-1},
\]
where $Df_i \in \Q[T]$ by Lemma \ref{le:bound}.  Since $\Q[T]$-coordinates of elements of $K(T)$ with respect to the power basis of $\beta$ are unique, we conclude that for $i=1,\ldots, d-1$, $f_i=\frac{c_i}{T^{\ell}}$, and $Df_i=\frac{Dc_i}{T^{\ell}} \in \Q[T]$.  Thus, 
\begin{equation}
\label{eq:ineq}
|\ord_{\qq_{\infty}}T^{\ell}| < |\ord_{\qq_{\infty}}Dc_i |< C(\beta)|\ord_{\qq_{\infty}}Y|,
\end{equation} 
or $b_i=0$.  If Inequality \eqref{eq:ineq} holds, then by Inequality \eqref{eq:order} we have that 
\[
|\ord_{\qq_{\infty}}Y^z|=|z\ord_{\qq_{\infty}}Y| <|\ell \ord_{\qq_{\infty}}T| < C(\beta) |\ord_{\qq_{\infty}}Y|,
\]
 so that $z<C(\beta)$.  The last inequality contradicts our assumptions on $z$.  Consequently, we have to conclude that $b_i=0$ for $i=1,\ldots,d-1$, and $Y\in \Q[T]$.  
 
We now return to Inequality \eqref{eq:order} and use the fact that we now know that $Y \in \Q[T]$.  We can therefore rephrase this inequality as saying 
\[
\deg {T^{\ell}} > \deg {Y^z} >\deg Y.
\]
  Thus from Equation \eqref{eq:equiv} we conclude that all coefficients of $Y$ are the same as the first $\deg Y$ coefficients of $M$.  However, $M \in \Z[T]$, and $M(0) =\pm 1$.  Hence the same must be true of $Y$.

We now assume that $Y \in \Z[T]$ and $Y(0)=1$.  Let $\ell > z\cdot \deg{Y}$.  Then $\ord_{\qq_{\infty}}\frac{T^{\ell}}{Y^z} <0$, and Inequality \eqref{eq:order} will be satisfied.  By Proposition \ref{forweak}, we can find $M \in \calD$ to satisfy Equation \eqref{eq:equiv}.  This completes the proof of the Claim.

A few quick observations  now complete the proof of the proposition.  First, we note that if a polynomial $R \in \Z[T]$, then there exists $c \in \Z$ such that $Y=R+c$ has its constant term equal to 1.  Second, we remind the reader that we have a Diophantine definition of elements of $k$ from Lemma \ref{le:containslocal}.  We also have a definition of non-constant elements of $R''$.  The non-constant elements must have a negative order at $\qq_{\infty}$.  Third, we remind the reader that by Lemma \ref{le:doubleprime}, the set $\{a, b \in R|a+bT \in R''\}$ is Diophantine over $R$.  Finally, we remind the reader that the set $\calD$ is Diophantine over $R''$ by Lemma \ref{le:special}.
\end{proof}

From this point on, to complete the proof of Theorem \ref{thm:extce}, we proceed as in the proof outline starting with Part \ref{it:start here}.

\section{Appendix}
This section contains some facts about roots of unity,  function fields  and real roots of polynomial equations collected here  for the convenience of the reader.  

\subsection{Roots of Unity}
Below, for $r \in \Z_{>0}$, the polynomial $\Phi_r(X) \in \Z[X]$ denotes the monic irreducible polynomial of a primitive $r$-th root of unity $\xi_r$.
\begin{lemma}
\label{ap1}
For every positive integer $t>1$, it is the case that $\Phi_t(0)=\pm 1$.
\end{lemma}
\begin{proof}
 Observe that $\Phi_t(X) \Big{|} (1+X +\ldots +X^{t-1})$ in $\Z[X]$, and therefore for some $U(X) \in \Z[X]$ we have that $\Phi_t(X)U(X)=1+X +\ldots +X^{t-1}$.  Hence, $\Phi_t(0)U(0)=1$, where $\Phi_t(0), U(0) \in \Z$.  So, $\Phi_t(0)=\pm 1$.
\end{proof}

\begin{lemma}
\label{le:approx}
Let $r \in \Z_{>0}$.  Let  $\ell=m_1\ldots m_r ,$ where $m_1, \ldots, m_r$ are pairwise relatively prime positive integers.  Then there exists  a set $\{p_1,\ldots, p_r\}$ of distinct prime numbers satisfying the following  conditions:
\begin{enumerate}
\item For all $i=1,\ldots,r$ we have that $p_i\equiv 1 \bmod \ell $.
\item For any $i \ne j \in \Z_{>0}$ such that $ \ell \equiv 0 \bmod j $ and $\ell \equiv 0 \bmod i$, it is the case that $(\Phi_j(\xi_i), p_i)=1$ in the ring of algebraic integers of $\Q(\xi_j)$.
\end{enumerate}
Further,  there exists $c \in \Z_{>0}$ such that $\ord_{p_i}\Phi_{m_i}(c)>0$, and for all $j | \ell, j \ne m_i$ we have that $\ord_{p_i}\Phi_j(c)=0$.
\end{lemma}
\begin{proof}
First of all, we note that the arithmetic sequence $(t\ell +1)_{t \in \Z_{>0}}$ contains infinitely many primes by the Dirichlet Density Theorem.  Therefore,  we can pick $p_1, \ldots, p_r$ so that none of these prime divides ${\mathbf N}_{\Q(\xi_i)/\Q}(\Phi_j(\xi_i))$ for  all $i, j, i \ne j$ dividing $\ell$.    \\

Since $p_i-1 \equiv 0 \bmod \ell$, for all $i \in \{1,\ldots,r\}$, by Hensel's Lemma, we have that a primitive root of unity $\xi_{s} \in \Q_{p_i}$- the field of $p_i$-adic numbers, for all positive integers $s$ dividing $\ell$.  In other words, $\Phi_s(T)$ splits completely in $\Q_{p_i}$.  Thus, there is an embedding of $\Q(\xi_s)$ into $\Q_{p_i}$ that maps the ideal generated by a factor of $p_i$ in the ring of integers of $\Q(\xi_s)$ into the ideal generated by $p_i$ in the ring of integers of $\Q_{p_i}$.  If $j \ne s$ for some positive integers $j, s$ dividing $\ell$,  then by assumption  $\Phi_{j}(\xi_s) $ is not contained in any ideal generated by a factor of $p_i$ in the ring of integers of $\Q(\xi_s)$.  Thus, in $\Q_{p_i}$, we have that $\ord_{p_i}(\Phi_j(\xi_s))=0$.  \\

For each $i=1,\ldots, r$ we pick $c_i \in \Z_{>0}$ such that $\ord_{p_i}(c_i-\xi_{m_i})>0$.   Observe that this choice of $c_i$ implies that
\[
\ord_{p_i}\Phi_{m_i}(c_i)= \ord_{p_i}(\Phi_{m_i}(c_i)-\Phi_{m_i}(\xi_{m_i})) \geq \ord_{p_i}(c_i -\xi_{m_i})>0.
\]
 At the same time, if $j \ne m_i$ and is a positive integer dividing $\ell$, then by assumption on $p_i$ we have that
 \[
 \ord_{p_i}\Phi_j(c_i)=\min(\ord_{p_i}(\Phi_j(c_i)-\Phi_j(\xi_{m_i})),\ord_{p_i}\Phi_j(\xi_{m_i}))=\ord_{p_i}\Phi_j(\xi_{m_i})=0.  
 \]

    By the Weak Approximation Theorem, we can find $c \in \Z$ such that 
    \[
    \ord_{p_i}(c-c_i)>\ord_{p_i}\Phi_{m_i}(c_i)>0.  
    \]
    Then $\ord_{p_i}\Phi_{m_i}(c)=\min(\ord_{p_i}\Phi_{m_i}(c)-\Phi_{m_i}(c_i), \Phi_{m_i}(c_i))=\ord_{p_i}\Phi_{m_i}(c_i)$.  Finally, if $j \ne m_i, j $ divides $\ell$, then we have 
    \[
    \ord_{p_i}(\Phi_j(c))=\min(\ord_{p_i}\Phi_{j}(c)-\Phi_{j}(c_i), \Phi_{j}(c_i))=\ord_{p_i}\Phi_{j}(c_i)=0.
    \]

\end{proof}
\subsection{Function Fields in One Variable}
\begin{proposition}
\label{prop:Chev}
Let $K/R$ be a finite extension of function fields.  Let $\pp$ be a prime (valuation) of $R$, and  let $\qq_1,\ldots, \qq_g$ be primes of $K$ lying above $\pp$ (or valuations of $K$ extending $\pp$).  Let $e_i=e(\qq_i/\pp)$ be the ramification index of $\qq_i$ over $\pp$ and let $f_i=f(\qq_i/\qq)$ be the relative degree of $\qq_i$ over $\pp$.  Then the following statements are true.
\begin{enumerate}
\item $\sum_{i=1}^ge_i f_i=[K:R]$. (See Chapter 4, \S1, Theorem 1 of \cite{C}.)
\item If $y \in K$, then $\ord_{\pp}{\mathbf N}_{K/R}(y)=f(\sum_{i=1}^g \ord_{\qq_i}y)$. (See  Chapter 4,  \S5, Corollary 2 (of Theorem 6) of \cite{C}).
\end{enumerate}
\end{proposition}

\begin{corollary}
\label{cor:one}
Let $M/k(T)$ be a function field extension, where $k(T)$ is a rational function field in $T$ over  a constant field $k$.  Suppose  in $M$ we have that $T$ has a pole at one prime $\qq_{\infty}$ only.  Let $\Qq_{\infty}$ be the infinite valuation of $k(T)$.  Then $\qq_{\infty}$ is the only prime of $M$ lying above $\Qq_{\infty}$, and $e(\qq_{\infty}/\Qq_{\infty})=\ord_{\qq_{\infty}}T$.
\end{corollary}

\begin{proof}
Suppose $\ttt \ne \qq_{\infty}$ is another prime of $M$ lying above $\Qq_{\infty}$.  Then $\ord_{\ttt}T <0$ in $M$, contradicting assumptions on $T$.  Further, $\ord_{\qq_{\infty}}T=e(\qq_{\infty}/\Qq_{\infty})\ord_{\Qq_{\infty}}T$.

\end{proof}
\begin{proposition}
\label{prop:oneonly}

\end{proposition}

\subsection{Real Roots}
\begin{proposition}
\label{prop:realroots}
Let $f(t) \in \Q[t]$ be a polynomial of degree $d$ irreducible over $\Q$ such that all of its roots are real. Let $\alpha_1< \ldots <\alpha_d$ be all the roots of $f(t)$ in $\R$.  Let $P_r(t) \in \Q(\alpha_r)[t]$.  Then there is an algorithm to determine wether $P_r(t)$ has any real roots for any $r \in \{1,\ldots,d\}$.
\end{proposition}
\begin{proof}
We can write $P_r(t)=\sum_{i=0}^{\deg {P_r(t)}-1}A_{r,i}t^i$, where $A_{r,i} \in \Q(\alpha_r)$.  Since $A_{r,i}=\sum_{i=0}^{d-1}a_{r,i,j}\alpha_r^j$, where $a_{r,i,j} \in \Q$,  we can now rewrite 
\[
P_r(t)=\sum_{i=0}^{\deg {P_r(t)}-1}\left (\sum_{j=0}^{d-1}a_{r,i,j}\alpha_r^j t^i\right )=Q_r(t, \alpha_r) \in \Q[t, \alpha_r].
\]
Fix $r \in \{1,\ldots, d\}$ and consider the following system
\begin{equation}
\label{eq:real}
\left \{
\begin{array}{c}
f(x_1)=0,\\
\ldots \\
f(x_{d})=0,\\
x_1-x_2=u_{1,2}^2 +v_{1,2}^2 +w_{1,2}^2 +z_{1,2}^2,\\
x_2-x_3=u_{2,3}^2 +v_{2,3}^2 +w_{2,3}^2 +z_{2,3}^2,\\
\ldots \\
x_{d-1} -x_d =u_{d-1,d}^2 +v_{d-1,d}^2 +w_{d-1,d}^2 +z_{d-1,d}^2,\\
Q_r(t,x_r)=0.
\end{array}
\right .
\end{equation}
We claim that  System \ref{eq:real} has a solution $(\beta_1,\ldots, \beta_d,\mu) \in \R^{d+1}$ if and only if $P_r(t)$ has a root in $\R$.  Indeed, suppose the system has solutions in $\R$.  Then $x_i=\alpha_i, i=1,\ldots,d$, and $Q_r(t,x_r)=P_r(t)$ has a real root.    Conversely,  suppose $P_r(t)$ has a real root $\beta_r$.  Then we can set $x_i =\alpha_i$, and $t=\beta_r$ to obtain a solution for the system.

By a result of A. Tarski (\cite{Tar93}) there is an algorithm to decide whether System \eqref{eq:real} has real solutions.
\end{proof}
\begin{corollary}
\label{cor:negative min}
Let $f(t) \in \Q(t), \alpha_1,\ldots,\alpha_d$ be as in Proposition \ref{prop:realroots}.  Let $P_r(t) \in \Q(\alpha_r)$.  Assume further that $\deg{P_r(t)}$ is even and the leading coefficient is positive.  Then there is an algorithm to determine whether there exists $\gamma_r \in \R$ such that $P_r(\gamma_r) \leq 0$.
\end{corollary}
\begin{proof}
Since $\lim_{t\rightarrow \pm \infty}P_r(t)=\infty$, if there exists $\gamma_r \in \R$ such that $P_r(\gamma_r) \leq 0$, then $P_r(t)$ has real roots.  Conversely, if $P_r(t)$ has real roots, then for some $\gamma_r \in \R$ we have that $P_r(\gamma_r)\leq 0$.
\end{proof}

\vskip .5in
\parbox{4.7in}{
{\sc
\noindent
Department of Mathematics \hfill \\
\hspace*{.1in}  Queens College -- C.U.N.Y. \hfill \\
\hspace*{.2in}  65-30 Kissena Blvd. \hfill \\
\hspace*{.3in}  Queens, New York  11367 U.S.A. \hfill \\
Ph.D. Programs in Mathematics \& Computer Science \hfill \\
\hspace*{.1in}  C.U.N.Y.\ Graduate Center\hfill \\
\hspace*{.2in}  365 Fifth Avenue \hfill \\
\hspace*{.3in}  New York, New York  10016 U.S.A. \hfill}\\
\hspace*{.045in} {\it E-mail: }
\texttt{Russell.Miller\at {qc.cuny.edu} }\hfill \\
\medskip
\hspace*{.045in} {\it Web page: }\texttt{qcpages.qc.cuny.edu/$\widetilde{~}$rmiller}\\
}
\vskip 1in
\parbox{4.7in}{
{\sc
\noindent
Department of Mathematics\hfill \\
\hspace*{.1in}  East Carolina University \hfill \\
\hspace*{.2in}  Greenville, NC 27858 U.S.A. \hfill}\\
\hspace*{.045in} {\it E-mail: }
\texttt{shlapentokha\at {ecu.edu} }\hfill \\
\hspace*{.045in} {\it Web page: }\texttt{myweb.ecu.edu/shlapentokha}
}


\begin{thebibliography}{20}

 \bibitem{C} Chevalley, C., {\em Introduction to the theory of algebraic functions of one
  variable},
      Mathematical Surveys, AMS,
   Providence, RI,
     1951,  vol. 6,

\bibitem{Da3} Davis, M.,
       On the number of solutions of {D}iophantine equations,
       Proc. Amer. Math. Soc. 35, 1972, 562-564
        

\bibitem{Demeyer10}  Demeyer, J., ``Diophantine sets of polynomials over number fields'', Proc. Amer. Math. Soc. 138 (2010), no. 8, 2715-2728.

\bibitem{Den5} Denef, J.,
     ``The diophantine problem for polynomial rings of positive
  characteristic'',
     1979,
   Logic colloquium 78,
     Boffa, M.,
    van Dalen, D.,
      MacAloon, K. editors,
   North Holland,
       131 \ndash 145.
       
\bibitem{Den0}   Denef, J.,
Diophantine sets over {${\bf Z}[T]$},
Proc. Amer. Math. Soc., 69(1), 1978, 148-150
              
\bibitem{Eispadic}   Eisentr\"ager, K., ``Hilbert's {T}enth {P}roblem for function fields of varieties over
  number fields and p-adic fields'', 2007, Journal of Algebra, 2007, volume 310,   775\ndash 792.
 


\bibitem{K-R3} Kim, H.,
  Roush, F.~W.,
    ``Diophantine unsolvability over $p$-adic function fields'',
      1995,
Journal of Algebra,
     176,
      83\ndash 110.

\bibitem{L} Lang, S., {\em Algebraic number theory},
Addison Wesley, Reading, MA, 1970.


\bibitem{Mat10} Matiyasevich, Y.,
     ``Towards finite-fold {D}iophantine representations'',
    2010,  Zap. Nauchn. Sem. S.-Peterburg. Otdel. Mat. Inst. Steklov.
  (POMI),
     377,
   Issledovaniya po Teorii Chisel. 10,
78\ndash 90.
        
\bibitem{Mat77} Matiyasevich, Y.,
     ``Existence of non-effectivizable estimates in the theory of
  exponential diophantine equations'',
 1977,
   Journal of Soviet Mathematics,
  8(3),  299\ndash  311,


\bibitem{MB3} Moret-Bailly, L.,
       ``Elliptic curves and {H}ilbert's {T}enth {P}roblem for algebraic
  function fields over real and $p$-adic fields'',
       2006,
  Journal f\"{u}r Reine und Angewandte Mathematic,
     587,
      77\ndash 143,

\bibitem{Pourchet71} Pourchet, Y.,
     ``Sur la repr\'esentation en somme de carr\'es des polyn\^omes \`a
  une ind\'etermin\'ee sur un corps de nombres alg\'ebriques'',
  1971,
       Acta Arith., 9,  89\ndash 104
     


\bibitem{PS} Poonen, B.
     and Shlapentokh, A.,
     ``Diophantine definability of infinite discrete non-archimedean
  sets and diophantine models for large subrings of number fields'',
2005,
   Journal f{\"u}r die Reine und Angewandte Mathematik,
     588, 27\ndash 48.


\bibitem{Sh34} Shlapentokh, A.,
      {\em  Hilbert's tenth problem: Diophantine classes and extensions to
  global fields},
  Cambridge University Press,
   2006.


\bibitem{Sh90}
   Shlapentokh, A.,
       ``Diophantine definitions for some polynomial rings'',
       1990,
  Comm. Pure Appl. Math.,
      43(8), 1055\ndash 1066.

\bibitem{Sh92}  Shlapentokh, A., ``Hilbert's tenth problem for rings of algebraic functions in one variable over fields of constants of positive characteristic'', Trans. Amer. Math. Soc. 333 (1992), no. 1, 275-298.
   
\bibitem{Smory}
      Smory{\'n}ski, C.,
       A note on the number of zeros of polynomials and exponential
  polynomials, J. Symbolic Logic, 42(1), 1977, 99-106.
  
  \bibitem{Soare87} Soare, R. I.,
     Recursively enumerable sets and degrees, Perspectives in Mathematical Logic,
      A study of computable functions and computably generated sets, Springer-Verlag, Berlin,1987
    
\bibitem{Tar93} Tarski, A., 
A decision method for elementary algebra and geometry. Quantifier elimination and cylindrical algebraic decomposition (Linz, 1993), Texts Monogr. Symbol. Comput., Springer, Vienna, 1998, 24-84. 
       
\bibitem{Zahidi2000} Zahidi, K.,
       On {D}iophantine sets over polynomial rings,
       Proc. Amer. Math. Soc. 128(3), 2000, 877-884.
        
 

\end{thebibliography}
\end{document}